\UseRawInputEncoding
\documentclass[12pt]{article}
\usepackage{amsmath}
\usepackage{amsthm}
\usepackage[colorlinks,  backref=page]{hyperref}
\usepackage{hyperref}
\usepackage{varioref}\usepackage{cleveref}
\usepackage{amssymb}
\theoremstyle{plain}
\newtheorem{thm}{Theorem}[section]
\newtheorem{cor}{Corollary}[section]
\newtheorem{lem}{Lemma}[section]

\theoremstyle{definition}

\newtheorem{rem}{Remark}[section]

\usepackage{titlesec}
\renewcommand{\b}[1]{\mathbf{#1}}

\renewcommand{\r}[1]{\mathrm{#1}}

\renewcommand{\(}{\left(}
\renewcommand{\)}{\right)}

\renewcommand{\leq}{\leqslant}
\renewcommand{\geq}{\geqslant}

\newcommand{\be}{\b e}

\newcommand{\bm}{\b m}

\newcommand{\sep}{\r{sep}}

\def\var{\varphi}

\DeclareMathSymbol{\twoheadrightarrow} {\mathrel}{AMSa}{"10}
\newcommand{\norm}[1]{\left\lVert#1\right\rVert}%norm%
\newcommand{\Rmn}[1]{\uppercase\expandafter{\romannueral#1}}%uppercase roman number%

\numberwithin{equation}{subsection}

\makeatletter
\newcommand{\fakephantomsection}{%
	\Hy@MakeCurrentHref{\@currenvir. \the\Hy@linkcounter}
	\Hy@raisedlink{\hyper@anchorstart{\@currentHref}\hyper@anchorend}%
}
\makeatother

\renewcommand{\det}{\mathrm{det}}
\allowdisplaybreaks

\def\b{\beta}

\def\({\left (}
\def\){\right )}
\def\<{\langle}
\def\>{\rangle}

\newcommand{\bel}[1]{\begin{equation}\label{#1}}
	\newcommand{\lab}[1]{\label{#1}}
	\newcommand{\beq}{\begin{equation}}

			\newcommand{\ea}{\end{eqnarray}}
		\newcommand{\qe}{\end{equation}}
	\newcommand{\eeq}{\end{equation}}

\newcommand{\lemref}[1]{Lemma~\ref{#1}}

\title{Symphonic map from ellipsoid to ellipsoid}\author{Xiangzhi Cao \thanks{School of information engineering,  Nanjing Xiaozhuang University,  Nanjing { 211171},  China. }}
	\newcommand{\enabstractname}{Abstract}
\newcommand{\cnabstractname}{摘要}

\numberwithin{equation}{section}

\begin{document}
	\maketitle
%\tableofcontents

\begin{abstract}
	In this paper,  we proved the existence of  Symphonic map from ellipsoid to ellipsoid via ellpsoid join method. We also give sufficient condition under which there exists nonconstnat solution. We also  give Hopf construction of Symphonic map from ellipsoid to ellipsoid via ellpsoid join method. 
\end{abstract}

	\section{Backgroud}\label{ddd}
		We firstly fix some notations and conventions used in this paper.  We denote $ \mathbb{S}^{n}  $  the unit sphere of dimension $n$,  we denote $ \mathbb{S}^{n}(r)  $  the sphere of dimension $n$ with radius $r$,  the metric $ g_{\mathbb{S}^n} $ is refered to as the standard induce metric of $ \mathbb{S}^n $ .       It is well known that	a map $ f: (M, g)\to (\mathbb{S}^n ,  g_{\mathbb{S}^n})$ is harmonic map if and only if	 Euler-Lagrangian equation
	$$
	\begin{aligned}
		\Delta_g f+|\nabla f|^2f=0
	\end{aligned}
	$$
	holds. 
 When the domain manifold $ M $ is $ S^m $,  using the above Euler-Lagrangian equation,   we can construct harmonic map $ f: (\mathbb{S}^m, g)\to  (\mathbb{S}^n ,  g_{\mathbb{S}^n})  $ by restricting  the map between Euclidean space,  whose component is  harmonic homogeneos polynomials
%	For a  map $u: \mathbb{S}^p \rightarrow \mathbb{S}^r$ ,  it is well known that if its components are harmonic $k$-齐次多项式,  then  $u$ is a harmonic map.  Its associated eigenvalue is $\lambda_u=k(k+p-1)$.  It is easy to calculate that $|d u(x)|^2=\lambda_u$ for all $x \in \mathbb{S}^p$.  

	%\fullcite{MR0143217}
%在1960年代,  Toda \cite{MR0143217} 给出球面同伦群的计算细节. 在1970年代,  Smith \cite{MR298590} 使用联合的办法考虑球面间的调和映照. 
% Peng 和 Tang\cite{MR1432425} 构造了球面同伦群的调和表示.  Tang \cite{MR1480875, MR2351423}考虑了同伦群$ \pi_{16p+j}(\mathbb{S}^{16p+j})$  的调和表示 .   Ratto \cite{MR1044657, MR1076254, MR1022743} 用ODE方法构造球面或者椭球面之间的调和映照.   Ding \cite{MR962492}  在一定的条件下获得球面间对称调和映照的存在性.   
In 1960s,  Toda \cite{MR0143217} gived complete details of these calculations of the homotopy group of spheres.  In 1970s,  Smith \cite{MR298590} considered harmonic map between spheres using join method. 
Peng and Tang\cite{MR1432425} construct harmonic representation of homotopy group of spheres.  Tang \cite{MR1480875, MR2351423} consider harmonic representation of the homotopy group $ \pi_{16p+j}(\mathbb{S}^{16p+j})$ under some conditons on $j$.   Ratto \cite{MR1044657, MR1076254, MR1022743} used ODE method to constructed harmonic map between sphere or ellipsoid.   Ding \cite{MR962492}  obtained the existence of symmetric harmonic map between sphere under some conditons.   
% Eells and Ferreira  \cite{MR1122903},   
 
%众所周知球面上的点  $x \in \mathbb{S}^{m_{1}+m_{2}-1}$ 可以唯一地表示为 $\left(x_{1} \sin t,  x_{2} \cos t\right)$,  其中 $x_{1} \in \mathbb{S}^{m_{1}-1},  x_{2} \in \mathbb{S}^{m_{2}-1},  t \in[0,  \pi / 2]$. 	
%给定 $f: \mathbb{S}^{m_{1}-1} \times \mathbb{S}^{m_{2}-1} \rightarrow S^{n-1}$ 和实值函数 $ \varphi(t) $,   $ f $的 Hopf构造 记作  $ h_{f}: \mathbb{S}^{m_{1}+m_{2}-1} \rightarrow \mathbb{S}^{n}, $ 它的定义如下
%\[
%h_{f}\left(x_{1} \sin t,  x_{2} \cos t\right):=\left(f\left(x_{1},  x_{2}\right) \sin \varphi(t) ,  \cos  \varphi(t)\right). 
%\]	  
It is well known that we can parametrize  $x \in \mathbb{S}^{m_{1}+m_{2}-1}$ uniquely  by $\left(x_{1} \sin t,  x_{2} \cos t\right)$ with $x_{1} \in \mathbb{S}^{m_{1}-1},  x_{2} \in \mathbb{S}^{m_{2}-1},  t \in[0,  \pi / 2]$. 	
Given a map $f: \mathbb{S}^{m_{1}-1} \times \mathbb{S}^{m_{2}-1} \rightarrow S^{n-1}$  and real function $ \varphi(t) $,  the Hopf construction for $ f $ is given by  $ h_{f}: \mathbb{S}^{m_{1}+m_{2}-1} \rightarrow \mathbb{S}^{n}, $ which is defined as
\[
h_{f}\left(x_{1} \sin t,  x_{2} \cos t\right):=\left(f\left(x_{1},  x_{2}\right) \sin \varphi(t) ,  \cos  \varphi(t)\right). 
\]

%令 $ \lambda_1 $ 和$\lambda_2 $为$f$的特征值. 在 $ m_1=m_2=2, \lambda_1=\lambda_2. $的情形之下,  Smith \cite{MR391127}证明了调和联合的存在性 .  如果$ 4\lambda_1<(m_1-2)^2,  4\lambda_2<(m_2-2)^2$ 或者 $ m_1=m_2=2, \lambda_1=\lambda_2 $,  Ratto \cite{MR1014467} 给出Hopf构造. 如果 $ \lambda_1 \leq  (m_1 -1)\lambda_2,  m_1 \geq3, m_2=2 $,  Ratto\cite{MR987757} 证明调和Hopf构造的不存在性 .  如果 $ \lambda_1>(m_1-1)\lambda_2, m_1\geq 3, m_2=2 $,  Gastel \cite{MR2022387} 证明了Hopf构造的存在 .   当 $m_1,  m_2 \geq 3$, Ding \cite{MR1298998}  证明了调和Hopf映照的存在性.  Later,  Gastel \cite{MR1985449} 简化了Ding \cite{MR1298998}中的证明.  后来,  如果 $ m_1=2, m_2>2,  \lambda_1 \geq 1, \lambda_2 >\lambda_1(m_2-1) $,  Ding,  Fan and Li \cite{MR2001704} 得到了Hopf构造的存在性.  
	 
	 Let $ \lambda_1 $ and $\lambda_2 $  be the eigenvalues of $ f $.  Smith \cite{MR391127} proved that  a harmonic Hopf construction exists in the case that $ m_1=m_2=2, \lambda_1=\lambda_2. $  Ratto \cite{MR1014467} gived hopf construction  in the case that $ 4\lambda_1<(m_1-2)^2,  4\lambda_2<(m_2-2)^2$ or $ m_1=m_2=2, \lambda_1=\lambda_2 $.  Ratto\cite{MR987757} proved that no harmonic Hopf  construction can exist if $ \lambda_1 \leq  (m_1 -1)\lambda_2,  m_1 \geq3, m_2=2 $.   Gastel \cite{MR2022387} proved  Hopf construction exists if $ \lambda_1>(m_1-1)\lambda_2, m_1\geq 3, m_2=2 $.  Ding \cite{MR1298998}  proved that there is always  a harmonic Hopf construction in the case $m_1,  m_2 \geq 3$.  Later,  Gastel \cite{MR1985449}  simplified the  proof of Ding’s result in \cite{MR1298998}.  Later,  Ding,  Fan and Li \cite{MR2001704} obtained the existence of Hopf construction if $ m_1=2, m_2>2,  \lambda_1 \geq 1, \lambda_2 >\lambda_1(m_2-1) $.

% $ f: (M, g)\to (N, h)$   Nakauchi引入交响映照的概念,  它是如下泛函
%$$
%\begin{aligned}
%	F(f)=\int_M \frac{\|f^* h \|^2}{2} \mathrm{d}v_g. 
%\end{aligned}
%$$
%的临界点. 
%在本文,  我们使用$ F(f) $ 表示 $ f $的能量.  选取$\{e_i\}_{i=1}^{m}$
%为关于$(M, g)$的局部幺正标架,  $\sigma_f(X)=\sum_j h\left(d f(X),  d f\left(e_j\right)\right) d f\left(e_j\right) $ ,     $$\operatorname{div}_g \sigma_f=\sum_j\left(\nabla_{e_j} \sigma_f\right)\left(e_j\right). 
%$$
%泛函$ F(f) $的 Euler-Lagrange 方程是 (cf.  \cite{MR2781758, MR2852338}) 
%\begin{equation}\label{cg}
%	\begin{aligned}
%		\operatorname{div}_g \sigma_f=0. 
%	\end{aligned}
%\end{equation}
%
%
%%=\sum_j \nabla_{e_j}\left(\sigma_f\left(e_j\right)\right)-\sum_j \sigma_f\left(\nabla_{e_j} e_j\right). 

Let $ f:(M, g)\to (N, h) $ be a map between Riemannian manifold.  In 2011,  Nakauchi and Takenaka \cite{MR2852338} introduced the concept of  symphonic map,  which is the critical point of the functional
$$
\begin{aligned}
	F(f)=\int_M \frac{\|f^* h \|^2}{2} dv_g. 
\end{aligned}
$$
In the sequel of this paper,  we also use the notation $ F(f) $ to denote the energy of the map $ f $. 

Let $ \{e_i\}_{i=1}^{m} $ be the local orthonormal frame,  $\sigma_f(X)=\sum_j h\left(d f(X),  d f\left(e_j\right)\right) d f\left(e_j\right) $ ,    and $$\operatorname{div}_g \sigma_f=\sum_j\left(\nabla_{e_j} \sigma_f\right)\left(e_j\right). 
$$
Its Euler-Lagrange equation of $ F(f) $ is as follows (cf.  \cite{MR2781758, MR2852338}) 
\begin{equation}\label{cg}
	\begin{aligned}
		\operatorname{div}_g \sigma_f=0. 
	\end{aligned}
\end{equation}

%=\sum_j \nabla_{e_j}\left(\sigma_f\left(e_j\right)\right)-\sum_j \sigma_f\left(\nabla_{e_j} e_j\right). 

In 2014,  Nakauchi and Takakuwa  \cite{MR3238295}  used symphonic joining  method to  construct symphonic map between spheres.    Inspired by \cite{MR3238295} and the related results in the field of harmonic map ,  The main purpose of this paper   are to construct symphonic map between ellipsoid (see \autoref{thm3. 1}) and gives Hopf construction of symphonic map(see \autoref{thm4. 1}).  One can refer to  \cite{MR3940323, MR3813997, MR3451407, MR4483578,  MR4434687, MR3989749, MR4484437} for the related results for symphonic map.  We think our main theorem (\autoref{thm3. 1}\autoref{thm4. 1}) can be used to represent the  the elements in homotopy group of ellipsoid by symphonic map.

%    2014年,  Nakauchi 和合作者 Takakuwa  \cite{MR3238295}  通过交响联合的办法来构造球面间的交响映照.   受到文献\cite{MR3238295} 和调和映照相关结果的启发 ,  本文的主要目的是构造椭球面之间的交响映照 ( \autoref{thm3. 1}) 以及 给出交响映照的Hopf构造( \autoref{thm4. 1}).  读者可以参考 \cite{MR3940323, MR3813997, MR3451407, MR4483578,  MR4434687, MR3989749, MR4484437} 关于交响映照的相关研究.  我们的定理 (\autoref{thm3. 1}, \autoref{thm4. 1}) 可以用来构造椭球面的同伦群中的元素. 
    
     The main  contribution of this work lies in establishing the ordinary differential equations that the symphonic  mappings satisfy (Lemma \ref{lem2} and Lemma \ref{dd}).  To our knowledge,  this is the first work concerning symphonic mappings that either originate from or into ellipsoids.  As for the proof of the existence part of Hopf construction,  it essentially adapts the arguments from the existing literature on harmonic mappings to the context of symphonic  mappings.

	This paper is organized as follows.  In Section \ref{cc},  we provide some preliminary lemmas that will be utilized in Sections 3 and 4.  Subsequently,  we state and prove \autoref{thm3. 1}.  In Section \ref{sec4},  we present and prove \autoref{thm4. 1}. 
%	本文的结构是: 在第 \ref{cc}节,  我们给出一些引理.  在第\ref{sec3}节,  我们叙述 \autoref{thm3. 1}的证明.  在第 \ref{sec4}节,  我们给出  \autoref{thm4. 1} 的证明. 

	\section{Preliminary Knowledge}\label{cc}
%	\begin{equation*}
%		\begin{split}
%			E_m(f)= \frac{1}{m}\int_M |f^{*}h|^m dv
%		\end{split}
%	\end{equation*}
%m-symphonic map is defined by
%	\begin{equation*}
%		\begin{split}
%			div(|f^{*}h|^{m-2}\sigma_f)=0
%		\end{split}
%	\end{equation*}
%which is equivalent to 
%\begin{equation*}
%	\begin{split}
%	div(\sigma_f)+ d(\log|f^{*}h|^{m-2})\cdot \sigma_f=0
%	\end{split}
%\end{equation*}
%which is equivalent to
%
%\begin{equation*}
%	\begin{split}
%		\Delta^{f^* h} f^\alpha+\left\|f^* h\right\|^2 f^\alpha+ \frac{m-2}{2}d(\log|f^{*}h|^{2})\cdot \sigma_f=0
%	\end{split}
%\end{equation*}

%\begin{equation*}
%	\begin{split}
%		\begin{aligned}
%			& -\left(m_1 \frac{\sin t}{\cos t}-m_2 \frac{\cos t}{\sin t}\right) \varphi^{\prime}(t)^3 \cos \varphi(t)+\left\|\left(f_1 * f_2\right)^* g_{\mathbb{S}^{n_1}(1) * \mathbb{S}^{n_2(1)}}\right\|^2 \sin \varphi(t)=0 \\
%			&
%		\end{aligned}
%	\end{split}
%\end{equation*}

%we will write $Q^n(a,  b)$ for $Q^{p+q+1}(a,  b)$,  slurring over the important dependence on the decomposition $p+q+1=n$. We denote  by $b / a$  the dilatation of $Q^{m_1+m_2+1}(a,  b)$, 
We consider ellipsoid (cf.   \cite{MR1044657})
$$Q^{m_1+m_2+1}(a,  b)=\{(x, y)\in \mathbb{R}^{m_1+1}\times  \mathbb{R}^{m_2+1}\mid \frac{|x|^2}{a^2}+\frac{|y|^2}{b^2}=1 \} . $$
It is well known that $Q^{m_1+m_2+1}(a,  b)$ is not generally isometric to $\mathbb{S}^{m_1+m_2+1}$. 
It is well known that the ellipsoid $Q^{m_1+m_2+1}(a,  b)$ is parametrized by
$$
z=a\cos t  \cdot x+b \sin t \cdot y, 
$$
for $x \in \mathbb{S}^{m_1},  y \in \mathbb{S}^{m_2}, z\in Q^{m_1+m_2+1}(a,  b)$ and $0 \leq t \leq \pi / 2$.  Some literature swaps the positions of $\cos(t)$ and $\sin(t)$ (cf.  \cite[p. ~4]{MR1044657}),  which has no impact on the final conclusion. 
The induced Riemannian metric on $Q^{m_1+m_2+1}(a,  b)$ is:
\begin{equation}\lab{metric}
	g=\left(a^2 \cos ^2 t\right) g_{\mathbb{S}^{m_2}}+\left(b^2 \sin ^2 t\right) g_{\mathbb{S}^{m_1}}+h^2(t) d t^2, 
\end{equation}
where 
$$
h(t)=\left[a^2 \sin ^2 (t) +b^2 \cos ^2 (t)\right]^{1 / 2}. 
$$
Its volume density is
$$
v_g=a^{m_1} b^{m_2} \left( \cos ^{m_1} t \sin ^{m_2} t\right)  h(t) \cdot v_{\mathbb{S}^{m_1}} \cdot v_{\mathbb{S}^{m_2}}, 
$$
where $v_{\mathbb{S}^{m_1}},  v_{\mathbb{S}^{m_2}}$ are volume densities of the indicated Euclidean spheres.

% $Q^{m_1+m_2+1}(a,  b)$ 上面的点可以参数化为
%$$
%z=a\cos t  \cdot x+b \sin t \cdot y, 
%$$
%其中 $x \in \mathbb{S}^{m_1},  y \in \mathbb{S}^{m_2}, z\in Q^{m_1+m_2+1}(a,  b)$,  $0 \leq t \leq \pi / 2$.  

% $Q^{m_1+m_2+1}(a,  b)$ 上的诱导度量为:
%\begin{equation}\label{metric}
%	g=\left(a^2 \cos ^2 t\right) g_{\mathbb{S}^{m_1}}+\left(b^2 \sin ^2 t\right) g_{\mathbb{S}^{m_2}}+h^2(t) d t^2, 
%\end{equation}
%其中
%$$
%h(t)=\left[a^2 \sin ^2 (t) +b^2 \cos ^2 (t)\right]^{1 / 2}. 
%$$
%体积密度
%$$
%v_g=a^{m_1} b^{m_2} \left( \cos ^{m_1} (t) \sin ^{m_2}( t)\right)  h(t) \cdot v_{\mathbb{S}^{m_1}} \cdot v_{\mathbb{S}^{m_2}}, 
%$$
%其中 $v_{\mathbb{S}^{m_1}},  v_{\mathbb{S}^{m_2}}$ 为单位球面的体积密度. 

%Also we will write
%$$
%v=a^p b^q \cos ^p t \sin ^q t h(t)
%$$

We  refer to $\left(Q^{m_1+m_2+1}(a,  b),  g\right)$ as an ellipsoid join of the spheres $\mathbb{S}^{m_1}$ and $ \mathbb{S}^{m_2}$.  It should be pointed out that ellipsoid join  is slightly different from Riemannian join  defined  in \cite[Definition 1]{MR3238295}.  We keep the noation with \cite{MR3238295} partially.  Note that some coefficients in this proof differ from those in \cite{MR3238295},  owing to the different definitions . 

%It should be pointed out that $\mathrm{SO}(m_1+1) \times \mathrm{SO}(m_2+1)$ is a group of isometries of $Q^{p+q+1}(a,  b)$ .  Thus $Q^{m_1+m_2+1}(a,  b)$ and $Q^{m_2+m_1+1}(b,  a)$ are isometric. 

%
%我们回忆从黎曼流形到椭球的调和映照的相关结果. 
% \begin{lem}[cf.  \cite{MR1044657, MR1022743}]\label{cccc} 	
% 设 $ \varphi $ 从黎曼流形 $ (M, g) $ 到椭球 $Q^n(c,  d)$的调和映照.   设 $\Phi=i \circ \varphi$ 为 $\varphi$ 和$i$ 的复合 , 其中$i$ 为 $Q^n(c,  d)$ 到 $\mathbb{R}^{n+1}$的典范嵌入.  那么 $\Phi$ 为调和映照当且仅当(cf. \cite[公式(2. 4)]{MR1044657}) 
% 	\begin{equation}\label{df}
% 		\Delta_g \Phi=\left(\left|P^{-\frac{1}{2}} d \Phi\right|^2 /\left|P^{-1} \Phi\right|^2\right) P^{-1} \Phi, 
% 	\end{equation}
%其中
% 	
% \begin{equation}\label{58}
% 	\begin{aligned}
% 		P=\left(\begin{array}{lllllll}
% 			c^2 & & & & & & 0 \\
% 			& \ddots & & & & & \\
% 			& & c^2 & & & & \\
% 			& & & d^2 & & & \\
% 			& & & & & \ddots & \\
% 			0 & & & & & &d^2
% 		\end{array}\right)
% 	\end{aligned}. 
% \end{equation}
% 	
% 	
% 
% 	
% 	
% 如果记 $\Phi=\left(\Phi_1,  \Phi_2\right)$,  为外围空间 $\mathbb{R}^{n+1}$ 投射到 $Q^n(c,  d)$的分量,  那么 \cref{df} 就是 (cf.  \cite{MR1044657, MR1022743})
% 	$$
% 	\left\{\begin{array}{l}
% 		\Delta \Phi_1=\left(\Lambda / c^2\right) \Phi_1 \\
% 		\Delta \Phi_2=\left(\Lambda / d^2\right) \Phi_2
% 	\end{array}\right. 
% 	$$
% 其中
% 	$$
% 	\Lambda=\frac{\left|d \Phi_1\right|^2 / c^2+\left|d \Phi_2\right|^2 / d^2}{\left|\Phi_1\right|^2 / c^4+\left|\Phi_2\right|^2 / d^4}. 
% 	$$
% 
% \end{lem}

We recall the related results of harmonic map into ellipsoid.  
\begin{lem}[cf.  \cite{MR1044657, MR1022743}]\label{cccc}
	
	Let $ \varphi $ be an harmonic map from Riemannian manifold $ (M, g) $ into ellipsoid $Q^n(c,  d). $  Let $\Phi=i \circ \varphi$ be the composition of $\varphi$ with the canonical embedding $i$ of $Q^n(c,  d)$ into $\mathbb{R}^{n+1}$.  Then $\Phi$ is harmonic iff 
	\begin{equation}\label{df}
		\Delta_g \Phi=\left(\left|P^{-\frac{1}{2}} d \Phi\right|^2 /\left|P^{-1} \Phi\right|^2\right) P^{-1} \Phi, 
	\end{equation}
	where
	
	\begin{equation}\label{58}
		\begin{aligned}
			P=\left(\begin{array}{lllllll}
				c^2 & & & & & & 0 \\
				& \ddots & & & & & \\
				& & c^2 & & & & \\
				& & & d^2 & & & \\
				& & & & & \ddots & \\
				0 & & & & & &d^2
			\end{array}\right)
		\end{aligned}. 
	\end{equation}

	If we write $\Phi=\left(\Phi_1,  \Phi_2\right)$,  the components being the projections on the factors of the ambient space $\mathbb{R}^{n+1}$ following the join construction of $Q^n(c,  d)$,  then \cref{df} becomes the system (cf.  \cite{MR1044657, MR1022743})
	$$
	\left\{\begin{array}{l}
		\Delta \Phi_1=\left(\Lambda / c^2\right) \Phi_1 \\
		\Delta \Phi_2=\left(\Lambda / d^2\right) \Phi_2
	\end{array}\right. 
	$$
	with
	$$
	\Lambda=\frac{\left|d \Phi_1\right|^2 / c^2+\left|d \Phi_2\right|^2 / d^2}{\left|\Phi_1\right|^2 / c^4+\left|\Phi_2\right|^2 / d^4}. 
	$$
	Such harmonic maps are real analytic. 
\end{lem}

%给定两个映照 $f_1: \mathbb{S}^{m_1} \rightarrow \mathbb{S}^{n_1}$ 和 $f_2: \mathbb{S}^{m_2} \rightarrow \mathbb{S}^{n_2}$,  我们可以用它们来构造映照 $f_1 * f_2$, 我们称之为椭球联合映照,  它是椭球面之间的映照
%$$
%\phi=f_1*f_2: Q^{m_1+m_2+1}(a,  b) \rightarrow Q^{n_1+n_2+1}(c,  d). 
%$$
%对于连续函数 $\varphi:[0,  \pi / 2] \rightarrow[0,  \pi / 2]$ with $\varphi(0)=0,  \varphi(\pi / 2)=\pi / 2$,  我们可以定义一个新映照
Given two maps $f_1: \mathbb{S}^{m_1} \rightarrow \mathbb{S}^{n_1}$ and $f_2: \mathbb{S}^{m_2} \rightarrow \mathbb{S}^{n_2}$,  we consider their ellipsoid join $f_1 * f_2$,  which is a map between ellipsoids
$$
\phi=f_1*f_2: Q^{m_1+m_2+1}(a,  b) \rightarrow Q^{n_1+n_2+1}(c,  d). 
$$
Indeed,  for any continuous function $\varphi:[0,  \pi / 2] \rightarrow[0,  \pi / 2]$ with $\varphi(0)=0,  \varphi(\pi / 2)=\pi / 2$,  we can define\begin{equation}\label{eq4}
	\begin{aligned}
	\phi(z)=f_1*f_2(z)=c \cos \varphi(t) \cdot f_1(x)+d\sin \varphi(t)  \cdot f_2(y), 
	\end{aligned}
\end{equation}
where  $x \in \mathbb{S}^{m_1},  y \in \mathbb{S}^{m_2}$,  $z=a\cos t  \cdot x+b \sin t \cdot y$和 $0 \leq t \leq \pi / 2$.  In the sequel of this paper ,  we all assume that $m_1,  m_2, n_1, n_2 \geq 1$ .

%\begin{lem}[]
%设$ f:(M, g)\to (N, h) $ 为交响映照,   $\mathbb{A}_N(X, Y)$ 为 $(N, h)$在 $ \mathbb{R}^{n} $的第二基本形式, 选取关于度量$ g$的幺正标架$\{e_i\}_{i=n}^m$,  那么我们有
% \begin{equation*}
% 	\begin{aligned}
% 		\Delta^{f^*h}f=\left( f^*h\right)^{ij}A_N(df(e_i), df(e_j)),  
% 	\end{aligned}
% \end{equation*}
%其中  $\Delta^{f * h}$ 的定义为
%\begin{equation}\label{cc12}
%	\begin{aligned}
%		\Delta^{f^* h} f^\alpha & = \frac{1}{\sqrt{\det(g)}}\sum_{i,  j=1}^m \frac{\partial}{\partial x^j}\left(\sqrt{\det(g)}\left(f^* h\right)^{i j} \frac{\partial f^\alpha}{\partial x^i}\right) \\
%		\left(f^* h\right)^{i j} & =\sum_{k,  \ell} g^{i k} g^{j \ell}\left(f^* h\right)\left(e_k,  e_{\ell}\right) . 
%	\end{aligned}
%\end{equation}
%使用算子 $ \operatorname{div}_g $的定义,  容易知道
%\begin{equation*}
%	\begin{aligned}
%		\operatorname{div}_g \sigma_f=0, 
%	\end{aligned}
%\end{equation*}
%等价于 
%\begin{equation*}
%	\begin{aligned}
%			\Delta^{f^*h}f=\left( f^*h\right)^{ij}A_N(df(e_i), df(e_j)) . 
%	\end{aligned}
%\end{equation*}
%我们可以照搬调和映照情形的证明到这里. 
%
%\end{lem}

From \cref{cg},  we can get 
\begin{lem}[]
	Let $ f:(M, g)\to (N, h) $ be an symphonic map,  let $A_N(X, Y)$ be the second fundmental form of (N, h) into $ \mathbb{R}^{n} $,  then we have 
	\begin{equation*}
		\begin{aligned}
			\Delta^{f^*h}f=\left( f^*h\right)^{ij}A_N(df(e_i), df(e_j)),  
		\end{aligned}
	\end{equation*}
	where $\Delta^{f * h}$ denotes the "Laplacian with respect to $f^* h$ ",  i. e. , 
	\begin{equation}\label{cc12}
		\begin{aligned}
			\Delta^{f^* h} f^\alpha & = \frac{1}{\sqrt{\det(g)}}\sum_{i,  j=1}^m \frac{\partial}{\partial x^j}\left(\sqrt{\det(g)}\left(f^* h\right)^{i j} \frac{\partial f^\alpha}{\partial x^i}\right) \\
			\left(f^* h\right)^{i j} & =\sum_{k,  \ell} g^{i k} g^{j \ell}\left(f^* h\right)\left(e_k,  e_{\ell}\right) . 
		\end{aligned}
	\end{equation}
	
	\end{lem}\begin{rem}
	When the target manifold $ (N, h)=(\mathbb{S}^n, g_{\mathbb{S}^n}) $,  it is just the formula (5) in \cite{MR3238295}.  In the definition of $\Delta^{f * h}$ in \cite{MR3238295} ,  the author may miss the factor $\sqrt{\det(g)}$, but the result in that paper is right. 
	\end{rem}
	\begin{proof}
		Using the definition of the operator $ \operatorname{div}_g $,  it is not hard to see that the equation
		\begin{equation*}
			\begin{aligned}
				\operatorname{div}_g \sigma_f=0, 
			\end{aligned}
		\end{equation*}
		is 
		\begin{equation*}
			\begin{aligned}
				\Delta^{f^*h}f=\left( f^*h\right)^{ij}A_N(df(e_i), df(e_j)) . 
			\end{aligned}
		\end{equation*}
		One can carry over the proof in the case of harmonic map to the present case. 
	\end{proof}
\begin{lem}
	Let $\Phi=i \circ \phi$ be the composition of $\phi$ with the canonical embedding $i$ of $Q^n(c,  d)$ into $\mathbb{R}^{n+1}$.  Then $\phi$ is symphonic map iff 
	\begin{equation}\label{125}
		\begin{aligned}
			\Delta^{\Phi^{*}h} \Phi=\left(\left|P^{-1 / 2}\Phi^{*}h \right|^2 /\left|P^{-1} \Phi\right|^2\right) P^{-1} \Phi, 
		\end{aligned}
	\end{equation}
	where  the notation $ \Delta^{f^{*} h} f^{\alpha} $ and $ \left(f^{*} h\right)^{i j} $ is as in Lemma 2. 2,  the matrix is defined in \eqref{58}. 
	%\begin{align}\label{cc12}
	%		 & =\sum_{i,  j=1}^{m} \frac{\partial}{\partial x^{j}}\left(\left(f^{*} h\right)^{i j} \frac{\partial f^{\alpha}}{\partial x^{i}}\right)  
	%	\end{align}
%
%and
%$$ \left(f^{*} h\right)^{i j}=\sum_{k,  \ell} g^{i k} g^{j \ell}\left(f^{*} h\right)\left(e_{k},  e_{\ell}\right) . $$
If we write $\Phi=\left(\Phi_1,  \Phi_2\right)$,  the components being the projections on the factors of the ambient space $\mathbb{R}^{n+1}$ following the join construction of $Q^n(c,  d)$,  then \eqref{125} becomes the system (cf.  \cite{MR1044657, MR1022743})
\begin{equation}\label{er}
	\begin{aligned}
		\Delta^{\Phi_1^{*}h} \Phi_1=\left(\Lambda / c^2\right) \Phi_1 \\
		\Delta^{\Phi_2^{*}h} \Phi_2=\left(\Lambda / d^2\right) \Phi_2
	\end{aligned}
\end{equation}
with
$$
\Lambda=\frac{\left|\Phi_1^{*}h\right|^2 / c^2+\left|\Phi_2^{*}h\right|^2 / d^2}{\left|\Phi_1\right|^2 / c^4+\left|\Phi_2\right|^2 / d^4}. 
$$
\end{lem}

\begin{lem}[cf. 	Lemma 1 in \cite{MR3238295}]\lab{eq9}
	Let $f$ be a smooth map from the manifold $ (M, g) $ into the standard sphere $\mathbb{S}^n$ of radius 1.  Let $f(x)=\left(f^1(x),  \ldots,  f^{n+1}(x)\right) \in \mathbb{S}^n \subset$ $\mathbb{R}^{n+1}$ and  $h=g_{\mathbb{S}^n}$.  Then $f$ is symphonic if and only if it satisfies the equation
	$$
	\Delta^{f^* h} f^\alpha+\left\|f^* h\right\|^2 f^\alpha=0 \quad(\alpha=1,  \ldots,  n+1), 
	$$
	
\end{lem}
%$$
%J(\alpha)=1 / 2 \int_0^{\pi / 2}\left[\frac{k^2(\alpha)}{h^2} \dot{\alpha}^2+\frac{c^2 \sin ^2 \alpha}{a^2 \sin ^2} \lambda_u+\frac{d^2 \cos ^2 \alpha}{b^2 \cos ^2} \lambda_v\right] v d s
%$$
%where $h=\left[b^2 \sin ^2+a^2 \cos ^2\right]^{1 / 2},  k(\alpha)=\left[d^2 \sin ^2 \alpha+c^2 \cos ^2 \alpha\right]^{1 / 2}$. 
\begin{lem}
	Let $\Phi=i \circ \phi$ be the composition of $\phi$ with the canonical embedding $i$ of $Q^n(c,  d)$ into $\mathbb{R}^{n+1}$.  Then $\Phi$ is symphonic map iff 
	\begin{equation}\label{125}
		\begin{aligned}
			\Delta^{\Phi^{*}h} \Phi=\left(\left|P^{-1 / 2}\Phi^{*}h \right|^2 /\left|P^{-1} \Phi\right|^2\right) P^{-1} \Phi, 
		\end{aligned}
	\end{equation}
	where  the notation $ \Delta^{f^{*} h} f^{\alpha} $ and $ \left(f^{*} h\right)^{i j} $ is as above,  the matrix is defined in \eqref{58}. 
	%\begin{align}\label{cc12}
	%		 & =\sum_{i,  j=1}^{m} \frac{\partial}{\partial x^{j}}\left(\left(f^{*} h\right)^{i j} \frac{\partial f^{\alpha}}{\partial x^{i}}\right)  
	%	\end{align}
%
%and
%$$ \left(f^{*} h\right)^{i j}=\sum_{k,  \ell} g^{i k} g^{j \ell}\left(f^{*} h\right)\left(e_{k},  e_{\ell}\right) . $$
If we write $\Phi=\left(\Phi_1,  \Phi_2\right)$ as the components being the projections on the factors of the ambient space $\mathbb{R}^{n+1}$ following the join construction of $Q^n(c,  d)$,  then \eqref{125} becomes the system (cf.  \cite{MR1044657, MR1022743})
\begin{equation}\label{er}
	\begin{aligned}
		\Delta^{\Phi_1^{*}h} \Phi_1=\left(\Lambda / c^2\right) \Phi_1 \\
		\Delta^{\Phi_2^{*}h} \Phi_2=\left(\Lambda / d^2\right) \Phi_2
	\end{aligned}
\end{equation}
with
$$
\Lambda=\frac{\left|\Phi_1^{*}h\right|^2 / c^2+\left|\Phi_2^{*}h\right|^2 / d^2}{\left|\Phi_1\right|^2 / c^4+\left|\Phi_2\right|^2 / d^4}. 
$$
\end{lem}

\begin{proof}
We follow the proof of \cite[Lemma 1]{MR3238295}.  	Let $ A $  be the second fundamental form of $Q^n(c,  d)$ in $\mathbb{R}^{n+1}$,  then
$$
\begin{aligned}
	\Delta^{\Phi^{*}h} \Phi=(\Phi^{*}h)^{ij}A(d\Phi(e_i), d\Phi(e_j)), 
\end{aligned}
$$
at $ (x, y) \in Q^n(c,  d) $,  by the computation (2. 15)-(2. 16) in \cite{MR3320921},  we have 
$$
\begin{aligned}
	A((X_1, Y_1), (X_2, Y_2))|_{(x, y)}=\frac{\langle X_1, X_2\rangle / c^2+\langle Y_1, Y_2\rangle / d^2}{\left|x\right|^2 / c^4+\left|y\right|^2 / d^4}\left(\frac{x}{c^2}, \frac{y}{d^2}	\right) 
\end{aligned}
$$

Thus, at $ (x, y)=(\Phi_1, \Phi_2) $, 
\begin{equation*}
	\begin{aligned}
		&(\Phi^{*}h)^{ij}A(d\Phi(e_i), d\Phi(e_j))\\
		&=(\Phi^{*}h)^{ij}\frac{\langle d\Phi_1(e_i), d\Phi_1(e_j)\rangle / c^2+\langle d\Phi_2(e_i), d\Phi_2(e_i)\rangle / d^2}{\left|\Phi_1\right|^2 / c^4+\left|\Phi_2\right|^2 / d^4}\left(\frac{\Phi_1}{c^2}, \frac{\Phi_1}{d^2}	\right) . 
	\end{aligned}
\end{equation*}
This implies \eqref{125}. 
Thus,  if we project the above equation according  to the join construction of ellipsoid  $Q^n(c,  d)$,  it is not hard to get \eqref{er} . 
\end{proof}\begin{lem}\label{lem2}
%设	$f_{1}: \mathbb{S}^{m_{1}} \longrightarrow \mathbb{S}^{n_{1}}$, 
%$f_{2}: \mathbb{S}^{m_{2}} \longrightarrow \mathbb{S}^{n_{2}}$ 为两个交响映照.  则如下两个性质等价:
%
%(1) 式\eqref{eq4}给出的联合 $f_{1} * f_{2}$ 为交响映照
%
%(2) 函数$ \varphi(t) $ 满足常微分方程
Let	$f_{1}: \mathbb{S}^{m_{1}} \longrightarrow \mathbb{S}^{n_{1}}$, 
$f_{2}: \mathbb{S}^{m_{2}} \longrightarrow \mathbb{S}^{n_{2}}$ be two symphonic maps.  Then the following two conditions are equivalent:

(1) The join $f_{1} * f_{2}$ given by \eqref{eq4} is a symphonic map. 

(2) $ \varphi(t) $ satisfies the ODE 
\begin{equation}\label{eq1}
	\begin{split}
		&\frac{k(\alpha)^2}{h(t)^2}(\varphi^{\prime}(t)^3)^{\prime}-	\frac{k(\alpha)^2}{h(t)^2}\left(m_1 \frac{\sin t}{\cos t}-m_2 \frac{\cos t}{\sin t}-2\frac{h^\prime(t)}{h(t)}\right)\varphi^{\prime}(t)^3+ \frac{2kk^\prime}{h(t)^2}\varphi^{\prime}(t)^4\\
		+&\left(\frac{c^4}{a^4}\frac{\cos ^2 \varphi(t)}{\cos ^4 t} \|f_1^* {g_{\mathbb{S}^{n_1}}}\|^2 -\frac{d^4}{b^4}\frac{\sin ^2 \varphi(t)}{\sin ^4 t} \|f_2^*{g_{\mathbb{S}^{n_2}}}\|^2 \right)  \sin \varphi(t)\cos \varphi(t)=0. 
	\end{split}
\end{equation}
where  $h(t)=\left[a^2 \sin^2(t)+b^2 \cos^2(t)\right]^{1 / 2},  k(\varphi)=\left[c^2 \sin ^2 \varphi+d^2 \cos ^2 \varphi\right]^{1 / 2}. $

(3).  The following equation  holds：
	$$\begin{aligned}
	\dfrac{d}{dt}\left(\frac{k^4(\var)}{h^4(t)}\left(\var^\prime\right)^3\right)+\frac{k^4(\var)}{h^4(t)}\left(\var^\prime\right)^3(-m_1\tan t+m_2\cot t)-\frac{k^3(\var)k^\prime(\var)(\var^\prime)^4}{h^4(t)}\\
	+a_1\frac{\cos^3\var\sin\var}{\cos^4t}-a_2\frac{\cos\var\sin^3\var}{\sin^4t}=0. 
\end{aligned}$$
where
\begin{equation}\label{cc-1}
	a_1=\frac{c^4}{a^4}\frac{ \mathcal{F}\left(f_1\right)}{\operatorname{Vol}\left(\mathbb{S}^{m_1}\right)},  \,  a_2=\frac{d^4}{b^4}\frac{ \mathcal{F}\left(f_2\right)}{\operatorname{Vol}\left(\mathbb{S}^{m_2}\right)} . 
\end{equation}
\end{lem}
\begin{rem}
Compared with that in \cite[Theorem 1]{MR3238295}, there is an extra term in this equation. 
\end{rem}
\begin{proof} We follow that in  \cite[Lemma 2 and Lemma  3]{MR1044657}.  We choose local coordinates$\{x^i\}_{i=1}^{m_1}$ on $\mathbb{S}^{m_1}$,  local coordinates $\{y^j\}_{j=1}^{m_2}$ on $\mathbb{S}^{m_2}$.   For convenience,  we adopt the notation of equation \eqref{cc12},  	
\begin{equation*}
	\begin{aligned}
		&\Delta^{(f_1*f_2)^*g_{Q^{n_1+n_2+1}(c, d)}}	\\ &=\frac{1}{\sqrt{\operatorname{det}\left(g_{Q^{m_{1}+m_{2}+1}(a,  b)}\right)}} \frac{\partial}{\partial x^i}\big\{\sqrt{\operatorname{det}\left(g_{Q^{m_{1}+m_{2}+1}(a,  b)}\right)} \frac{1}{a^2\cos^2 t} (g_{\mathbb{S}^{m_1}})^{ia}\frac{1}{a^2\cos^2 t} (g_{\mathbb{S}^{m_1}})^{jb} c^2 \cos^2\varphi(t) \\&(f_1^{*}g_{\mathbb{S}^{n_1}})_{ab}\frac{\partial}{\partial x^j}  \big\} \\
		&+\frac{1}{\sqrt{\operatorname{det}\left(g_{Q^{m_{1}+m_{2}+1}(a,  b)}\right)}} \frac{\partial}{\partial y^k}\big\{\sqrt{\operatorname{det}\left(g_{Q^{m_{1}+m_{2}+1}(a,  b)}\right)} \frac{1}{b^2\sin^2 t} (g_{\mathbb{S}^{m_1}})^{kc}\frac{1}{b^2\sin^2 t} (g_{\mathbb{S}^{m_2}})^{ld} c^2 \sin^2\varphi(t) \\&(f_1^{*}g_{\mathbb{S}^{n_1}})_{cd}\frac{\partial}{\partial y^l}  \big\} \\
		&+\frac{1}{\sqrt{\operatorname{det}\left(g_{Q^{m_{1}+m_{2}+1}(a,  b)}\right)}} \frac{\partial}{\partial t}\big\{\sqrt{\operatorname{det}\left(g_{Q^{m_{1}+m_{2}+1}(a,  b)}\right)} \frac{1}{h(t)^2}\varphi^{\prime}(t)^2 k^2(\varphi) \frac{\partial}{\partial t}  \bigg\}. 
	\end{aligned}
\end{equation*}
Thus
\begin{equation}\label{er}
	\begin{aligned}
			&\Delta^{(f_1*f_2)^*g_{Q^{n_1+n_2+1}(c, d)}}\\
			=& \frac{c^2\cos ^{2} \varphi(t)}{a^4\cos ^{4} t} \frac{1}{\sqrt{\operatorname{det}\left(g_{\mathbb{S}^{m_{1}}}\right)}} \sum_{k,  \ell} \frac{\partial}{\partial x^{i}}\left\{\sqrt{\operatorname{det}\left(g_{\mathbb{S}^{m_{1}}}\right)}\left(f_{1}^{*} g_{\mathbb{S}^{n_{1}}}\right)^{ij} \frac{\partial}{\partial y^{j}}\right\} \\
		&+ \frac{d^2\sin ^{2} \varphi(t)}{b^4\sin ^{4} t} \frac{1}{\sqrt{\operatorname{det}\left(g_{\mathbb{S}^{m_{2}}}\right)}} \sum_{k,  \ell} \frac{\partial}{\partial y^{k}}\left\{\sqrt{\operatorname{det}\left(g_{\mathbb{S}^{m_{2}}}\right)}\left(f_{2}^{*} g_{\mathbb{S}^{n_{2}}}\right)^{k \ell} \frac{\partial}{\partial y^{\ell}}\right\} \\
		 &+\frac{1}{\left(\frac{\cos t}{r_1}\right)^{m_1}\left(\frac{\sin t}{r_2}\right)^{m_2}} \frac{\partial}{\partial t}\left\{\left(\frac{\cos t}{r_1}\right)^{m_1}\left(\frac{\sin t}{r_2}\right)^{m_2}\frac{1}{h^{2}(t)} \varphi^{\prime}(t)^2 k^2(\varphi)\frac{\partial}{\partial t}\right\} \\
		 =&\frac{c^2}{a^4} \frac{\cos ^2 \varphi(t)}{\cos ^4 t} \Delta^{f_1^* g_{\mathbb{S}^n 1}}+\frac{d^2}{b^4} \frac{\sin ^2 \varphi(t)}{\sin ^4 t} \Delta^{f_2^* g_{\mathbb{S}^{n_2} }} \\
		 &+\frac{\partial}{\partial t}\left(\frac{1}{h^2(t)}\varphi^{\prime}(t)^2 k^2(\varphi) \frac{\partial}{\partial t}\right)-\left(m_1 \frac{\sin t}{\cos t}-m_2 \frac{\cos t}{\sin t}\right)\frac{1}{h^2(t)} \varphi^{\prime}(t)^2 \frac{\partial}{\partial t} 
	\end{aligned}
\end{equation}
This formula is similar to that in \cite[Lmma  3, fomula (e)]{MR1044657}. 

On the other hand
\begin{equation}\label{ddr}
	\begin{aligned}
		 &\left\|\left(\left(f_1 * f_2\right)\right)^* g_{\mathbb{S}^{n_1+n_2+1}}\right\|^2\\
		 =&\sum_{i,  j,  k,  \ell}\left(g_{Q^{m_{1}+m_{2}+1}(a,  b)}\right)^{i k}\left(g_{Q^{m_{1}+m_{2}+1}(a,  b)}\right)^{j \ell} \left(\left(f_1 * f_2\right)^* g_{\mathbb{S}^{n_1+n_2+1}}\right)_{i j}\left(\left(f_1 * f_2\right)^* g_{\mathbb{S}^{n_1+n_2+1}}\right)_{k \ell} \\
		 =&\sum_{a,  b,  c,  d}\left(\frac{1}{a^2\cos ^2 t}\left(g_{\mathbb{S}^{m_1}}\right)^{a c}\right)\left(\frac{1}{a^2\cos ^2 t}\left(g_{\mathbb{S}^{m_1}}\right)^{b d}\right) \\
		& \times\left(\frac{}{}c^2\cos ^2 \varphi(t)\left(f_1^* g_{\mathbb{S}^{n_1}}\right)_{a b}\right)\left(c^2\cos ^2 \varphi(t)\left(f_1^* g_{\mathbb{S}^{n_1}}\right)_{c d}\right) \\
		& +\sum_{p,  q,  r,  s}\left(\frac{1}{b^2\sin ^2 t}\left(g_{\mathbb{S}^{m_2}}\right)^{p r}\right)\left(\frac{1}{b^2\sin ^2 t}\left(g_{\mathbb{S}^{m_2}}\right)^{q s}\right) \\
		& \times\left(\frac{}{}d^2\sin ^2 \varphi(t)\left(f_2^* g_{\mathbb{S}^{n_2}}\right)_{p q}\right)\left(d^2\sin ^2 \varphi(t)\left(f_2^* g_{\mathbb{S}^{n_2}}\right)_{r s}\right)+\frac{k^4(\varphi(t))}{h^4(t)}\varphi^{\prime}(t)^4 \\
		 =&\left(\frac{c^4}{a^4}\right)\left(\frac{\cos ^2 \varphi(t)}{\cos ^2 t}\right)^2\left\|f_1^* g_{\mathbb{S}^{n_1}}\right\|^2 +\left(\frac{d^4}{b^4}\right)\left(\frac{\sin ^2 \varphi(t)}{\sin ^2 t}\right)^2\left\|f_2^* g_{\mathbb{S}^n}\right\|^2+\frac{k^4(\varphi(t))}{h^4(t)}\varphi^{\prime}(t)^4 .  \\
	\end{aligned}
\end{equation}

Since  $ f_{1} * f_{2} $ is assumed to symphonic map,  we have  (cf.  \cite[Theorem 1]{MR3238295} )
\begin{equation}\label{et}
	\begin{cases}
		\begin{aligned}
			\Delta^{\left(f_{1} * f_{2}\right)^{*} g_{\mathbb{S}^{n_{1}+n_{2}+1}}}\left(\cos \varphi(t) f_{1}(x)\right)+\frac{\Lambda}{c^2} \cos \varphi(t) f_{1}(x)=0 \\
			\Delta^{\left(f_{1} * f_{2}\right)^{*} g_{\mathbb{S}^{n_{1}+n_{2}+1}}}\left(\sin \varphi(t) f_{2}(x)\right)+ \frac{\Lambda}{d^2}  \sin \varphi(t) f_{2}(x)=0
		\end{aligned}
	\end{cases}	
\end{equation}

Gathereing \eqref{er}and  \eqref{et} gives 
\begin{equation}\label{ff}
	\begin{aligned}
		&\frac{c^2\cos ^2 \varphi(t)}{a^4\cos ^4 t} \Delta^{f_1^* g_{\mathbb{S}^{n_1} }}\bigg|_{\mathbb{S}^{m_1}}\cos\left( \varphi(t) f_1\right) +\frac{d^2\sin ^2 \varphi(t)}{b^4\sin ^4 t} \Delta^{f_2^* g_{\mathbb{S}^{n_2}}}\bigg|_{\mathbb{S}^{m_2}}(\cos\varphi(t) f_1) \\
		& +\frac{\partial}{\partial t}\left(\frac{1}{h^2(t)}\varphi^{\prime}(t)^2 \frac{\partial}{\partial t}\right)(\cos\varphi(t) f_1)-\left(m_1 \frac{\sin t}{\cos t}-m_2 \frac{\cos t}{\sin t}\right)\frac{1}{h^2(t)} \varphi^{\prime}(t)^2 \frac{\partial}{\partial t}(\cos\varphi(t) f_2)  \\	
		&+\frac{\Lambda}{c^2}\cos \varphi(t) f=0. 
	\end{aligned}
\end{equation}
Meanwhile
\begin{equation}\label{ff1}
	\begin{aligned}
		&\frac{c^2\cos ^2 \varphi(t)}{a^4\cos ^4 t} \Delta^{f_1^* g_{\mathbb{S}^{n_1} }}\bigg|_{\mathbb{S}^{m_1}}\sin\left( \varphi(t) f_2\right) +\frac{d^2\sin ^2 \varphi(t)}{b^4\sin ^4 t} \Delta^{f_2^* g_{\mathbb{S}^{n_2}}}\bigg|_{\mathbb{S}^{m_2}}(\sin\varphi(t) f_2) \\
		& +\frac{\partial}{\partial t}\left(\frac{1}{h^2(t)}\varphi^{\prime}(t)^2 \frac{\partial}{\partial t}\right)(\sin\varphi(t) f_2)-\left(m_1 \frac{\sin t}{\cos t}-m_2 \frac{\cos t}{\sin t}\right)\frac{1}{h^2(t)} \varphi^{\prime}(t)^2 \frac{\partial}{\partial t}(\sin\varphi(t) f_2)\\	
		&+\frac{\Lambda}{d^2}\sin \varphi(t) f_2=0. 
	\end{aligned}
\end{equation}
Since $f_1$ and $f_2$ are both symphonic map, which sastisfied the equation in \lemref{eq9} , using the similar trick in \cite{MR3238295},   It is not hard to see that   \eqref{ff}$\cdot c^2\sin(\varphi(t)) $-\eqref{ff1}$ \cdot d^2\cos(\varphi(t))=0 $ is just
\begin{equation}\label{eq11}
	\begin{split}
		&\frac{k(\alpha)^2}{h(t)^2}(\varphi^{\prime}(t)^3)^{\prime}-	\frac{k(\alpha)^2}{h(t)^2}\left(m_1 \frac{\sin t}{\cos t}-m_2 \frac{\cos t}{\sin t}-2\frac{h^\prime(t)}{h(t)}\right)\varphi^{\prime}(t)^3+ \frac{2kk^\prime}{h(t)^2}\varphi^{\prime}(t)^4\\
		+&\left(\frac{c^4}{a^4}\frac{\cos ^2 \varphi(t)}{\cos ^4 t} \|f_1^* {g_{\mathbb{S}^{n_1}}}\|^2 -\frac{d^4}{b^4}\frac{\sin ^2 \varphi(t)}{\sin ^4 t} \|f_2^{*}g_{\mathbb{S}^{n_2}}\|^2 \right)  \sin \varphi(t)\cos \varphi(t)=0. 
	\end{split}
\end{equation}
It is hard to see that  $ c^2\cos(\varphi(t))\cdot \eqref{ff} $+$d^2\sin(\varphi(t)) \cdot\eqref{ff1} =0 $  is equivalent to \eqref{ddr}.

At last,  we  prove (3).  In fact,  let  	$$\begin{aligned}
L(\var, \var^\prime, t)=\left(\frac{k^4(\varphi(t))}{h^4(t)}\varphi^{\prime}(t)^4+a_1 \frac{\cos ^4 \varphi(t)}{\cos ^4 t}+a_2 \frac{\sin ^4 \varphi(t)}{\sin ^4 t}\right) \cos ^{m_1} t \sin ^{m_2} t
\end{aligned}$$
Computing directly 	$$\begin{aligned}
	\frac{d}{dt}\left(\frac{\partial L}{\partial \var^\prime}\right)-\frac{\partial L}{\partial \var}=0. 
\end{aligned}$$
gives (3).  In fact,  it is easy to say that formula (2) is also the critical point of $J(\var)$ defined in  \eqref{589}. 
\end{proof}

In order to apply variation method,   inspired by the work \cite{MR1044657},   we need to consider the wighted Hilbert spacew
\begin{equation}\label{space}
	X=\left\{\alpha \in W_v^{1, 4}([0,  \pi / 2],  \mathbb{R}):\|\alpha\|_{W_v^{1, 4}}^4=\int_0^{\pi / 2}\left( \dot{\alpha}^4+\alpha^4 \right) v d s<\infty\right\}	
\end{equation}
and
	$$\begin{aligned}
	X_0=\{\var\in X|0\leq \var \leq \frac{\pi}{2} , t\in(0, \frac{\pi}{2} \}
\end{aligned}$$
where $ v $ is the function defined on $ [0, \frac{\pi}{2}] $ ,  $\mathrm{W}_v^{1,  4}\left(\left(0,  \frac{\pi}{2}\right),  \mathbb{R}\right)$ denotes wighted Sobolev space.

设$ \mathcal{F}\left(f_1\right)=\int_M\norm{{f_1^* g_{\mathbb{S}^{n_1}}}}^2 $,  $ \mathcal{F}\left(f_2\right)=\int_M |f_2^* g_{\mathbb{S}^{n_2}}|^2 $.  For given function $ \varphi(t) $,  the function $h(t) , k(\varphi)$ is as in Lemma \ref{lem2},  we need to consider the following functional naturally
\begin{equation}\label{589}
	\begin{split}
		J(\varphi)=\frac{1}{4}\int_0^{\frac{\pi}{2}}\left(\frac{k^4(\varphi(t))}{h^4(t)}\varphi^{\prime}(t)^4+a_1 \frac{\cos ^4 \varphi(t)}{\cos ^4 t}+a_2 \frac{\sin ^4 \varphi(t)}{\sin ^4 t}\right) \cos ^{m_1} t \sin ^{m_2} t d t
	\end{split}
\end{equation}
where $a_1, a_2$ are given by \eqref{cc-1}. 

In fact ,  let $g$ be given in \eqref{metric},  $
g^\prime=\left(c^2 \cos ^2 s\right) g_{\mathbb{S}^{n_1}}+\left(d^2 \sin ^2 s\right) g_{\mathbb{S}^{n_2}}+k^2(s) \mathrm{d} s^2
$. 
For map $\left(Q^{m_1+m_2+1}(a,  b), g\right)\to \left(Q^{n_1+n_2+1}(c,  d), g^\prime\right) $, a simple computation shows 
	$$\begin{aligned}
J(\varphi)=\int_{Q^{m_1+m_2+1}(a,  b)}\norm{\left(f_1*f_2\right)^* h}^2	\mathrm{d} \text{v}_g
\end{aligned}$$
\begin{rem}
		If  $ m_1<4$or $m_2<4$,   $J$ is allowed to take the value $+\infty$.  If $m_1,  m_2\geq 4, a\neq 0, b\neq 0$,  since $h$ has positive lower bound,  the functional  $J$ can be defined on$X$ and smooth. 
\end{rem}
\begin{rem}
% We have swapped $\cos (t)$and $\sin (t)$ in the parametrization  of the point in elllipoind in section 2, so this results in the difference between the functional $J(\varphi)$ and that of 	(\cite[(2. 6)]{MR1044657}) , but it doesn't matter.  
The parametrization of the ellipsoid in Section 2 differs from that in \cite[(2. 6)]{MR1044657} by an interchange of 
$\cos (t)$and $\sin (t)$.   Consequently,  the associated functional $J(\varphi)$  slightly,  though this has no substantive consequence. 
\end{rem}
For map $f_1*f_2$ defined in \eqref{eq4},  then 
\begin{equation*}
	\begin{split}
		\| d(f_1*f_2)\|^2=  \frac{c^2\cos^2\varphi(t)}{a^2\cos^2(t)} |df_1 |^2+\frac{d^2\sin^2\varphi(t)}{b^2\sin^2(t)} |df_2 |^2+\frac{k^2(\varphi(t))}{h^2(t)}\varphi^\prime(t)^2
	\end{split}
\end{equation*}

Following that of Lemma 4 in \cite{MR1044657} ,   A routine computation can give 
\begin{lem}\label{lem-9}
For the function $ \varphi(t) $ and the map $ f_1*f_2 $ given in\eqref{eq4} , 有
	\begin{equation*}
		\begin{aligned}
			F(f_1*f_2)=ab\operatorname{Vol}(\mathbb{S}^{m_1})\operatorname{Vol}(\mathbb{S}^{m_2}) J(\varphi)
		\end{aligned}
	\end{equation*}
\end{lem}

\begin{rem}
	By this Lemma,  symphonic map is also the critical poit of  $J(\varphi)$.   In order to  find the  ellipsoid join map,  just need  to find the  critical point  of $J(\varphi)$. 
\end{rem}

Similar to \cite[Lemma 2. 8]{MR1044657},  wo also have 
\begin{lem}
There exists constant  $C$ such that for $\alpha \in X$
	$$
	\left. \begin{array}{l}
		\int_0^{\pi / 2} \alpha^4(s) \sin ^{p-4} (s)\cos ^q (s) d s \\
		\int_0^{\pi / 2} \alpha^4 \sin ^p (s) \cos ^{q-4} (s )d s
	\end{array}\right\} \leq C \int_0^{\pi / 2}\left[\dot{\alpha}^4+\alpha^4\right] \sin ^p (s) \cos ^q (s) d s . 
	$$

\end{lem}
\begin{proof}
This follows from Sobolev inequality on ([0, $\frac{\pi}{2}$],  $ \sin^p(s) \cos^{q-4}(s) ds $) and ([0, $\frac{\pi}{2}$],  $ \sin^{p-4} (s)\cos^{q}(s) ds $).   
\end{proof}

\begin{lem}\label{33}
	There exists minimizer  $\widehat{\varphi}(\mathrm{t})$ of $J(\varphi)$ in $X_0$ which $\widehat{\varphi}(\mathrm{t})$ satisifies Euler-Lagrange equation.  (\ref{eq1}). 
\end{lem}

\begin{proof}
%使用辅助空间
%$$
%\mathcal{A}=\left\{\var\in X \mid \varphi(0)=0,  \varphi\left(\frac{\pi}{2}\right)=\frac{\pi}{2}\right\} . 
%$$
%Take a minimizing sequence $\varphi_i$ in $\mathcal{A}$,  L. e.  $\varphi_T \in \mathcal{A}$ such that $\mathcal{F}(\varphi) \longrightarrow$ inf $f_{s \in A}\{(\varphi)$.  Then since
%there exists a weak limit $\varphi_{x,  i}$ i. e. 
%$$
%\varphi_i \longrightarrow \varphi_{\infty} \text { in } A \text { as } i \longrightarrow \infty
%$$
%Then since the functional $\mathcal{F}$ is lower semicontinuous,  we have
%$$
%\inf _{\varphi \in A} \mathcal{F}(\varphi) \leq \mathcal{F}\left(\varphi_{\infty}\right) \leq \liminf _{h \rightarrow \infty} \mathcal{F}\left(\varphi_3\right) \leq \inf _{\varphi \in A} \mathcal{F}(\varphi)
%$$
%Therefore 

%在$X$中选取极小化序列 $\varphi_i$  ,    得到一个弱极限$\varphi_{\infty}$.  由于泛函$F$ 是下半连续的, 因此泛函$J$ 也是下半连续的,  $\varphi_{\infty}$ 是 $X$中的极小化元 .  有Ding的方法 (cf.   \cite[引理2. 1]{MR962492}),   我们可以得到一个弱极限 $\varphi_{\infty}$ , 其也是空间$X_0$
%$$
%\tilde{\mathcal{A}}=\left\{\varphi \in \mathrm{W}^{1. 4}\left(\left(0,  \frac{\pi}{2}\right)\right) \mid \varphi(0)=0,  \varphi\left(\frac{\pi}{2}\right)=\frac{\pi}{2}\right\} . 
%$$

We follow that of \cite{MR3238295}.   Take a minimizing sequence $\varphi_i$ in $X$,    we get a weak limit $\varphi_{\infty}$.  Since the functional $J$ is lower semicontinuous,  $\varphi_{\infty}$ is a minimizer in $X$ .  By Ding's method (cf.  Lemma 2. 1 in \cite{MR962492}),   we get a weak limit   $\varphi_{\infty}$ which  is also a minimizer in $X_0$ and therefore $\varphi_{\infty}$ solves (\ref{eq1}) ,  set $\hat{\varphi}=\varphi_{\infty}$ and  the proof is completed. 
\end{proof}
\begin{rem}
	Here the constant solutions  $\widehat{\varphi}(\mathrm{t})\equiv 0 $ or $\widehat{\varphi}(\mathrm{t})\equiv \frac{\pi}{2}$ is not excluded(cf. \cite{MR962492}).  
\end{rem}

Furthermore,  we give sufficient condition under which there exists different solution of (\ref{eq1}) from $\widehat{\varphi}(\mathrm{t})\equiv 0 $ or $\widehat{\varphi}(\mathrm{t})\equiv \frac{\pi}{2}. $.  

Set $c_0:=\inf \{J(\var)| \var\in X,  0\leq \var \leq \frac{\pi}{2} , t\in(0, \frac{\pi}{2})\}$
\begin{cor}\label{r-1}
If $c_0< \min\{J(0), J(\frac{\pi}{2})\}$,  then $\widehat{\varphi}(\mathrm{t})$ satisfies  Euler-Lagrange equation $(\cref{eq2})$
\end{cor}
A routine computation shows that 
\begin{lem}\label{r-2}  
	The first variation  of $J(\varphi)$ in the    direction  $\xi$ is
	$$\begin{aligned}
		dJ(\varphi)(\xi)&\\
		=& 4\int_{0}^{\frac{\pi}{2}}\left( \frac{k^3(\varphi(t))\left(\varphi^\prime(t)\right)^4 \xi+k^4(\varphi(t))(\varphi^\prime(t))^3\xi^\prime}{h^4}\right. \\
		&\left. +\left( a_2 \frac{\sin^3(\varphi)\cos(\varphi)}{\sin^4(t)}-a_1 \frac{\cos^3(\varphi)\sin(\varphi)}{\cos^4(t)}\right)\xi\right) \cos ^{m_1} t \sin ^{m_2} t \mathrm{d}t 
	\end{aligned}$$
\end{lem}

\begin{lem}\label{r-3} The second variation  of $J(\varphi)$ in the    direction $\xi$	
	$$\begin{aligned}
		\nabla^2J(\varphi)(\xi, \xi)=&\int_{0}^{\frac{\pi}{2}} \left\{ \frac{\left(3k(\var)(\var^\prime)^4\xi^2+4k^3(\var)\left(\var^\prime\right)^3 \xi\xi^\prime\right) +4k^3(\var) (\var^\prime)^3\xi\xi^\prime+3k(\var)^4 (\var^\prime)^2\left(\xi^\prime\right)^2}{h^4}\right. \\
		&+a_2\frac{3\sin
			^2\var\cos^2 \var -\sin^4 \var}{\sin^4 t}\xi^2\\
		&\left. -a_1 \frac{-3\cos^2\var\sin^2\var +\cos^4\var }{\cos^4 t} \xi^2 \right\} \cos ^{m_1} t \sin ^{m_2} t  \mathrm{d}t
	\end{aligned}$$

\end{lem}

	\begin{proof}A direct computation shows that the second variation  of $J(\varphi)$ in the    direction  $\eta, \xi$  is
	$$\begin{aligned}
		\nabla^2J(\varphi)(\eta, \xi)=&\int_{0}^{\frac{\pi}{2}} \left\{ \frac{\left(3k(\var)\eta(\var^\prime)^4+4k^3(\var)\left(\var^\prime\right)^3 \eta^\prime\right) \xi+4k^3(\var)\eta (\var^\prime)^2\eta^\prime\xi^\prime}{h^4}\right. \\
		&+a_2\frac{3\sin
			^2\var\cos^2 \var \eta-\sin^4 \var \eta}{\sin^4 t}\xi\\
		&\left. -a_1 \frac{-3\cos^2\var\sin^2\var \eta+\cos^4\var \eta}{\cos^4 t} \xi \right\} \cos ^{m_1} t \sin ^{m_2} t  \mathrm{d}t
	\end{aligned}$$
\end{proof}

\begin{cor}
If $J(0)\leq J(\frac{\pi}{2})$,  then $c_0 <J(0)$ ,  thus  there exists nonconstant solution of  (\cref{eq1}) in $X_0$. 
\end{cor}

\begin{proof}
	$$\begin{aligned}
	I(\xi)=	\nabla^2J(0)(\xi, \xi)=\int_{0}^{\frac{\pi}{2}}\left(-a_1 \frac{1 }{\cos^4 t} \xi^2 \right) \cos ^{m_1} t \sin ^{m_2} t dt 
\end{aligned}$$
Obviously,  there exists function $\xi\geq 0$ in $X$ such that $I(\xi)<0$.  A similar arugument in (cf.   \cite[Lemma 2. 3]{MR962492}) shows that $c_0 <J(0). $
\end{proof}

\begin{lem}\label{r-4}For $p\geq 4, q\geq 4$,  then
	
	$$\begin{aligned}
	\int_{0}^{\frac{\pi}{2}} \sin^p s \cos^{q-4}s ds= \frac{(p-1)(p-3)}{(q-1)(q-3) }\int_{0}^{\frac{\pi}{2}} \sin^{p-4} s \cos^{q}s ds
\end{aligned}$$

\end{lem}

\begin{lem}\label{r-5}If   $a_1=\frac{c^4}{a^4}\frac{ \mathcal{F}\left(f_1\right)}{\operatorname{Vol}\left(\mathbb{S}^{m_1}\left(r_1\right)\right)},  \,  a_2=\frac{d^4}{b^4}\frac{ \mathcal{F}\left(f_2\right)}{\operatorname{Vol}\left(\mathbb{S}^{m_2}\left(r_2\right)\right)}, a_2\leq \frac{(m_2-1)(m_2-3)}{(m_1-5)(m_1-4)}a_1$,  
	 then  $J(0)\leq J(\frac{\pi}{2})$
\end{lem}
\begin{proof}
		$$\begin{aligned}
		J(\frac{\pi}{2})-J(0)=\frac{1}{4}\left(a_2-\frac{(m_2-1)(m_2-3)}{(m_1-5)(m_1-4)}a_1\right)\int_{0}^{\frac{\pi}{2}}\sin^{m_2-4}t\cos^{m_1}(t)dt
	\end{aligned}$$
\end{proof}

\begin{cor}
If  $a_1=\frac{c^4}{a^4}\frac{ \mathcal{F}\left(f_1\right)}{\operatorname{Vol}\left(\mathbb{S}^{m_1}\left(r_1\right)\right)},  \,  a_2=\frac{d^4}{b^4}\frac{ \mathcal{F}\left(f_2\right)}{\operatorname{Vol}\left(\mathbb{S}^{m_2}\left(r_2\right)\right)}, a_2\leq \frac{(m_2-1)(m_2-3)}{(m_1-5)(m_1-4)}a_1$,  then  there exists nonconstant solution of (\cref{eq2}) in $X_0$. 
\end{cor}

Similar to the case of harmonic map, we also need to consider the regularity of $\hat{\varphi}$. 

\begin{lem}\label{ccc}
	 $\widehat{\varphi} \in \mathrm{C}^{\infty}\left(\left(0,  \frac{\pi}{2}\right)\right)$
\end{lem}

\begin{proof}
	  Usingsobolev embedding  of  $\mathrm{W}^{1, 4}\left(\left(0,  \frac{\pi}{2}\right)\right)$  into $\mathrm{C}^0\left(\left[0,  \frac{\pi}{2}\right]\right)$,  we  can get  $\widehat{\varphi}(t) \in C^0\left(\left[0,  \frac{\pi}{2}\right]\right)$. 
set $\psi(t)=k^2(\widehat{\varphi})\left(\widehat{\varphi}^{\prime}(t)\right)^3$.  The functon $h(t) , k(\varphi)$ is as in Lemma \ref{lem2}.  By Lemma \ref{lem2}, 可得
$$
\psi^{\prime}(t)+A(t) \psi(t)+B(t)=0
$$
where
$$
\begin{aligned}
	\begin{cases}
		& A(t)=\frac{k(\alpha)^2}{h(t)^2}\left(m_1 \frac{\sin t}{\cos t}-m_2 \frac{\cos t}{\sin t}-2\frac{h^\prime(t)}{h(t)}\right) \\
		& B(t)=-\sin \widehat{\varphi}(t) \cos \widehat{\varphi}(t)\left(\frac{d^4}{b^4}\frac{\sin ^2 \widehat{\varphi}(t)}{\sin ^4 t}  \|f_1^* g_{g^n 1}\|^2-\frac{c^4}{a^4}\frac{\cos ^2 \widehat{\varphi}(t)}{\cos ^4 t}\| f_2^* g_{s^2 2(t)} \|^2\right) .  \\
	\end{cases}	
\end{aligned}
$$

%$$
%\begin{aligned}
%&the first line\\
%&\left\{\left.  \begin{aligned}
%	& A(t)=\frac{k(\alpha)^2}{h(t)^2}\left(m_1 \frac{\sin t}{\cos t}-m_2 \frac{\cos t}{\sin t}-2\frac{h^\prime(t)}{h(t)}\right) \\
%& B(t)=-\sin \widehat{\varphi}(t) \cos \widehat{\varphi}(t)\left(\frac{d^4}{b^4}\frac{\sin ^2 \widehat{\varphi}(t)}{\sin ^4 t}  \|f_1^* g_{g^n 1}\|^2-\frac{c^4}{a^4}\frac{\cos ^2 \widehat{\varphi}(t)}{\cos ^4 t}\| f_2^* g_{s^2 2(t)} \|^2\right) .  \\
%\end{aligned} \right. 
%\right)
%\end{aligned}
%$$
 Noticing that  $\hat{\varphi} \in \mathrm{W}_v^{1, 4}\left(0,  \frac{\pi}{2}\right)$,   $\int_{t_0-\varepsilon}^{t_0+\varepsilon}|A(t)|^4 d t<\infty$ ,  $\int_{t_0-\varepsilon}^{t_0+\varepsilon}|B(t)|^4 d t<\infty$. 
For fixed $t_0 \in\left(0,  \frac{\pi}{2}\right)$,  $t \in\left(0,  \frac{\pi}{2}\right)$,  as in \cite{MR3238295},  we get
$$
\begin{aligned}
	\left|\psi(t)-\psi\left(t_0\right)\right| & =\left|\int_{r_0}^t \psi^\prime(t) d t\right| \\
%	& =| \int_{t_0}^t(A(t) \psi(t)+B(t)) d t \mid \\
%	& \leq\left|\int_{t_0}^t\left(|A(t)|\left|\hat{\psi}^{\prime}(t)\right|^3+|B(t)|\right) d t\right| \\
	& \leq\left\{\int_{t_0}^t|A(t)|^4 d t\right\}^{\frac{1}{4}}\left\{\int_{t_0}^t\left|\hat{\varphi}^{\prime}(t)\right|^4 d t\right\}^{\frac{3}{4}}|+| \int_{t_0}^t|B(t)| d t \mid \\
	& \longrightarrow 0 \text { as } t \longrightarrow t_0 . 
\end{aligned}
$$
 Therefore  $\psi \in \mathrm{C}^0\left(\left(0,  \frac{\pi}{2}\right)\right)$.   Bootstrab argument shows that 
$$
\psi^{\prime}(t)=-A(t) \psi(t)-B(t)
$$
%Since the right hand side of $(13)$ is of class $\mathrm{C}^0,  \psi(t)$ is of class $\mathrm{C}^{-1}$,  therefore $\hat{\phi}(t)$ is of class $\mathrm{C}^2$.  Since the right hand side of $(13)$ is of class $\mathrm{C}^{\mathrm{l}},  \psi(t)$ is of class $\mathrm{C}^2$,  therefore $\bar{\psi}(t)$ is of class $\mathrm{C}^1$.  lnductively this argument 
If follows from here  $\hat{\varphi}(t)$ is of $\mathrm{C}^{\infty}$.

\end{proof} 
%正如\cite{MR3238295},   从上述引理可以得到 
%\begin{thm}\label{thm3. 1}
%给定两个交响映照
%	$$
%	\begin{aligned}
%		& f_1: \mathbb{S}^{m_1}\left(r_1\right) \longrightarrow \mathbb{S}^{n_1}\\
%		& f_2: \mathbb{S}^{\mathrm{m}_2}\left(r_2\right) \longrightarrow \mathbb{S}^{\mathrm{n}_2}
%	\end{aligned}
%	$$
%存在一个椭球联合交响映照
%	$$
%	f_1 * f_2: Q^{m_1+m_2+1}(a,  b)  \longrightarrow Q^{n_1+n_2+1}(a,  b)
%	$$
%%	ie. ,  a symphonic map which is a  ellipsoid join of maps $f_1$ andid $f_2$. 
%\end{thm}
The above lemmas imply the existence of the join of maps.  Just as in \cite{MR3238295},   we directly  get 
\begin{thm}\label{thm3. 1}
	For any rwo symphionic maps
	$$
	\begin{aligned}
		& f_1: \mathbb{S}^{m_1}\left(r_1\right) \longrightarrow \mathbb{S}^{n_1}\\
		& f_2: \mathbb{S}^{\mathrm{m}_2}\left(r_2\right) \longrightarrow \mathbb{S}^{\mathrm{n}_2}
	\end{aligned}
	$$
	there exists a symphonic ellipsoid join
	$$
	f_1 * f_2: Q^{m_1+m_2+1}(a,  b)  \longrightarrow Q^{n_1+n_2+1}(a,  b)
	$$
	ie. ,  a symphonic map which is a  ellipsoid join of maps $f_1$ andid $f_2$. 
\end{thm}
%Proof.  By Lemmas 4-7,  we have a smooth solution of Eq.  (7).  Then by Theorem 1,  the solution gives the existence of the join of $\operatorname{maps} f_1=f_2$
%As a corollary of Theorem 2 ,  we have the following result. 

\section{Hopf construction of symphonic map}\label{sec4}

%We call $b / a$ the dilatation of $Q^{p+q+1}(a,  b)$.  The ellipsoid $Q^{p+q+1}(a,  b)$ is parametrized by
%$$
%z=a \sin s \cdot x+b \cos s \cdot y
%$$
%for $x \in S^p,  y \in S^q$ and $0 \leq s \leq \pi / 2$. 
%The induced Riemannian metric on $Q^{p+q+1}(a,  b)$ is:
%$$
%g=\left(a^2 \cos ^2 t\right) g_p+\left(b^2 \sin ^2 t\right) g_q+h^2(t) d t^2
%$$
%where $g_p,  g_q$ denote the Euclidean metrics of $S^p,  S^q$ and
%$$
%h(s)=\left[b^2 \sin ^2 s+a^2 \cos ^2 s\right]^{1 / 2}
%$$

We consider the special ellipsoid $ N=\{(x_1, x_2, \cdots, x_{n+1})\in \mathbb{R}^{n+1}| \frac{x_1^2+x_2^2+\cdots x_n^2}{c^2}+\frac{x_{n+1}^2}{d^2}=1\} $,  $Q^{m_{1}+m_{2}+1}(a,  b)$ is given in section 2.   For map $f: \mathbb{S}^{m_{1}} \times \mathbb{S}^{m_{2}} \rightarrow \mathbb{S}^{n-1},  $the associated Hopf map (cf. \cite{MR1014467}) $ h_{f}:Q^{m_{1}+m_{2}+1}(a,  b) \rightarrow N$ ,  is defined as 
\begin{equation}\label{eq3}
	\begin{aligned}
		h_{f}\left(ax_{1} \cos t,  bx_{2} \sin t\right):=\left(cf\left(x_{1},  x_{2}\right) \cos \varphi(t),  d\sin \varphi(t)\right)
	\end{aligned}
\end{equation}
where $x \in \mathbb{S}^{m_{1}+m_{2}-1}$ can be parametrized by  $\left(x_{1} \sin s,  x_{2} \cos s\right)$ ,  $x_{1} \in \mathbb{S}^{m_{1}},  x_{2} \in \mathbb{S}^{m_{2}},  t \in[0,  \pi / 2]$.  

The metric $g$ on the ellipsoind $Q^{m_{1}+m_{2}+1}(a,  b)$ is
$$
g=\left(a^2 \cos ^2 t\right) g_{\mathbb{S}^{m_1}}+\left(b^2 \sin ^2 t\right) g_{\mathbb{S}^{m_2}}+h^2(t) d t^2
$$
the metric $g_N$ on the ellipsoind  $N$ is
$$
g_N=\left(c^2 \cos ^2 s\right) g_{\mathbb{S}^{n-1}}+\left(d^2 \sin ^2 s\right) g_{\mathbb{S}^{1}}+k^2(s) d s^2
$$
where
 $$
h(t)=\left[a^2 \sin ^2 (t)+b^2 \cos ^2 (t)\right]^{1 / 2},  k(s)=\left[c^2 \sin ^2 (s)+d^2 \cos ^2 (s)\right]^{1 / 2}
$$
$g_{\mathbb{S}^{m}}$is  the standard metric on the unit sphere $\mathbb{S}^{m}$. 
%This Hopf construction is a tool for

\begin{lem}\label{dd}
Given a map $f: \mathbb{S}^{m_{1}} \times \mathbb{S}^{m_{2}} \rightarrow \mathbb{S}^{n-1}$,  the associated Hopf map  $ h_{f}:Q^{m_{1}+m_{2}+1}(a,  b) \rightarrow  N $ is  given by \eqref{eq3}.  Assume that  $ f_1(\cdot)= f(\cdot, y), f_2(\cdot)= f(x, \cdot) $,  where $ x \in \mathbb{S}^{m_{1}},  y \in \mathbb{S}^{m_{2}} $.  Assume that $ f_1 $ and $ f_2 $ are both symphonic map.   Then the following two conditions are equivalent:

(1) The join $h_f$ is a symphonic map. 

(2) $ \varphi(t) $ satisfies the ordinary differential equation, 
%给定映照 $f: \mathbb{S}^{m_{1}} \times \mathbb{S}^{m_{2}} \rightarrow \mathbb{S}^{n-1}$,  伴随的Hopf映照 $ h_{f}:Q^{m_{1}+m_{2}+1}(a,  b) \rightarrow  N $由 \eqref{eq3}给出.  假设 $ f_1(\cdot)= f(\cdot, y), f_2(\cdot)= f(x, \cdot) $,  其中 $ x \in \mathbb{S}^{m_{1}},  y \in \mathbb{S}^{m_{2}} $.  设映照$ f_1 $ 和 $ f_2 $ 均为交响映照.   那么如下两个性质等价:
%
%	(1)  $h_f$ 为交响映照. 
%	
%	(2) $ \varphi(t) $ 满足常微分方程, 
	\begin{equation}\label{eq2}
		\begin{split}
			&\frac{k^2(\varphi)}{h(t)^2}(\varphi^{\prime}(t)^3)^{\prime}-	\frac{k^2(\varphi)}{h(t)^2}\left(m_1 \frac{\sin t}{\cos t}-m_2 \frac{\cos t}{\sin t}-2\frac{h^\prime(t)}{h(t)}\right)\varphi^{\prime}(t)^3+ \frac{2kk^\prime}{h(t)^2}\varphi^{\prime}(t)^4\\
			+	&\left(\frac{c^4\cos ^2 \varphi(t)}{a^4\cos ^4 t} \|f_1^* {g_{\mathbb{S}^{n-1}}}\|^2 +\frac{c^4\cos ^2 \varphi(t)}{b^4 \sin ^4 t} \|f_2^*{g_{\mathbb{S}^{n-1}}}\|^2 \right)  \sin \varphi(t)\cos \varphi(t)=0
		\end{split}
	\end{equation}
	
	(3).  The following equation  holds：
		$$\begin{aligned}			\frac{d}{dt}\left(\frac{k^4(\var)}{h^4(t)}\left(\var^\prime\right)^3\right)-\frac{k^3(\var)k^\prime(\var)(\var^\prime)^4}{h^4(t)}-\left(\frac{a_1}{\cos^4 t}+\frac{a_2}{\sin^4 t}\right)\sin^3(\var) \cos(\var)
	\end{aligned}$$
	where
\begin{equation}\label{cc-1}
		 a_1=\frac{c^4 \mathcal{F}\left(f_1\right)}{a^4\operatorname{Vol}\left(\mathbb{S}^{m_1}\right)},  a_2=\frac{c^4 \mathcal{F}\left(f_2\right)}{b^4\operatorname{Vol}\left(\mathbb{S}^{m_2}\right)} . 
\end{equation}
\end{lem}
\begin{proof}We choose local coordinates$\{x^i\}_{i=1}^{m_1}$ on $\mathbb{S}^{m_1}$,  local coordinates $\{y^j\}_{j=1}^{m_2}$ on $\mathbb{S}^{m_2}$.   We keep the same with the notation in \eqref{cc12},  we have 
\begin{equation}
	\begin{aligned}
			&\Delta^{h_f^*g_{N}}	\\ &=\frac{1}{\sqrt{\operatorname{det}\left(g_{Q^{m_{1}+m_{2}+1}(a,  b)}\right)}} \frac{\partial}{\partial x_i}\big\{\sqrt{\operatorname{det}\left(g_{Q^{m_{1}+m_{2}+1}(a,  b)}\right)} \frac{1}{a^2\cos^2 t} (g_{\mathbb{S}^{m_1}})^{ia}\frac{1}{a^2\cos^2 t} (g_{\mathbb{S}^{m_1}})^{jb} c^2 \cos^2\varphi(t) \\&(f^{*}g_{\mathbb{S}^{n-1}})_{ab}\frac{\partial}{\partial x^j}  \big\} \\
		&+\frac{1}{\sqrt{\operatorname{det}\left(g_{Q^{m_{1}+m_{2}+1}(a,  b)}\right)}} \frac{\partial}{\partial y^j}\big\{\sqrt{\operatorname{det}\left(g_{Q^{m_{1}+m_{2}+1}(a,  b)}\right)} \frac{1}{b^2\sin^2 t} (g_{\mathbb{S}^{m_2}})^{kc}\frac{1}{b^2\sin^2 t} (g_{\mathbb{S}^{m_2}})^{ld} c^2 \cos^2\varphi(t) \\&(f^{*}g_{\mathbb{S}^{n-1}})_{cd}\frac{\partial}{\partial y^j}  \big\} \\
		&+\frac{1}{\sqrt{\operatorname{det}\left(g_{Q^{m_{1}+m_{2}+1}(a,  b)}\right)}} \frac{\partial}{\partial t}\big\{\sqrt{\operatorname{det}\left(g_{Q^{m_{1}+m_{2}+1}(a,  b)}\right)} \frac{1}{h(t)^2}\varphi^{\prime}(t)^2  \frac{\partial}{\partial t}  \bigg\}
	\end{aligned}
\end{equation}
因此
\begin{equation}\label{}
	\begin{aligned}
			\Delta^{h_f^*g_{N}}	
		=& \frac{c^2\cos ^{2} \varphi(t)}{a^4\cos ^{4} t} \frac{1}{\sqrt{\operatorname{det}\left(g_{\mathbb{S}^{m_{1}}}\right)}} \sum_{k,  \ell} \frac{\partial}{\partial x^{i}}\left\{\sqrt{\operatorname{det}\left(g_{\mathbb{S}^{m_{1}}}\right)}\left(f^{*} g_{\mathbb{S}^{n-1}}\right)^{ij} \frac{\partial}{\partial x^{j}}\right\} \\
		&+ \frac{c^2\cos ^{2} \varphi(t)}{b^4\sin ^{4} t} \frac{1}{\sqrt{\operatorname{det}\left(g_{\mathbb{S}^{m_{2}}}\right)}} \sum_{k,  \ell} \frac{\partial}{\partial y^{k}}\left\{\sqrt{\operatorname{det}\left(g_{\mathbb{S}^{m_{2}}}\right)}\left(f^{*} g_{\mathbb{S}^{n-1}}\right)^{k \ell} \frac{\partial}{\partial y^{\ell}}\right\} \\
		& +\frac{1}{\left(a\cos t\right)^{m_1}\left(b\sin t\right)^{m_2}} \frac{\partial}{\partial t}\left\{\left(a\cos t\right)^{m_1}\left(b\sin t\right)^{m_2}\frac{1}{h^{2}(t)} \varphi^{\prime}(t)^2 \frac{\partial}{\partial t}\right\} \\
		& = \frac{c^2\cos ^2 \varphi(t)}{a^4\cos ^4 t} \Delta^{f^* g_{\mathbb{S}^{n-1} }}|_{\mathbb{S}^{m_1}}+\frac{c^2\cos ^2 \varphi(t)}{b^4\sin ^4 t} \Delta^{f^* g_{\mathbb{S}^{n-1}}}|_{\mathbb{S}^{m_2}} \\
		& +\frac{\partial}{\partial t}\left(\frac{1}{h^2(t)}\varphi^{\prime}(t)^2 \frac{\partial}{\partial t}\right)-\left(m_1 \frac{\sin t}{\cos t}-m_2 \frac{\cos t}{\sin t}\right)\frac{1}{h^2(t)} \varphi^{\prime}(t)^2 \frac{\partial}{\partial t} .  \\
		&
	\end{aligned}
\end{equation}
因此
\begin{equation}\label{df1}
	\begin{aligned}
 &\left\|h_f^*g_{N}\right\|^2\\
 =&\sum_{i,  j,  k,  \ell}\left(g_{Q^{m_{1}+m_{2}+1}(a,  b)}\right)^{i k}\left(g_{Q^{m_{1}+m_{2}+1}(a,  b)}\right)^{j \ell} \times\left(h_f^* g_{X}\right)_{i j}\left(h_f^*g_{X}\right)_{k \ell} \\
 =&\sum_{a,  b,  c,  d}\left(\frac{1}{a^2\cos ^2 t}\left(g_{\mathbb{S}^{m_1}}\right)^{a c}\right)\left(\frac{1}{a^2\cos ^2 t}\left(g_{\mathbb{S}^{m_1}}\right)^{b d}\right) \\
& \times\left(\frac{}{}c^2\cos ^2 \varphi(t)\left(f_1^* g_{\mathbb{S}^{n-1}}\right)_{a b}\right)\left(c^2\cos ^2 \varphi(t)\left(f_1^* g_{\mathbb{S}^{n-1}}\right)_{c d}\right) \\
& +\sum_{p,  q,  r,  s}\left(\frac{1}{b^2\sin ^2 t}\left(g_{\mathbb{S}^{m_2}}\right)^{p r}\right)\left(\frac{1}{b^2\sin ^2 t}\left(g_{\mathbb{S}^{m_2}}\right)^{q s}\right) \\
& \times\left(c^2\cos ^2 \varphi(t)\left(f_2^* g_{\mathbb{S}^{n-1}}\right)_{p q}\right)\left(c^2\cos ^2 \varphi(t)\left(f_2^* g_{\mathbb{S}^{n-1}}\right)_{r s}\right)+\frac{k^2(\varphi(t))}{h^4(t)}\varphi^{\prime}(t)^4 \\
 =&\left(\frac{c^2}{a^2\cos ^2 t}\right)^2\cos ^4 \varphi(t)\left\|f_1^* g_{\mathbb{S}^{n-1}}\right\|^2  +\left(\frac{c^2}{b^2\sin ^2 t}\right)^2\cos ^4 \varphi(t)\left\|f_2^* g_{\mathbb{S}^{n-1}}\right\|^2+\frac{k^2(\varphi(t))}{h^4(t)}\varphi^{\prime}(t)
	\end{aligned}
\end{equation}
Since Hopf map is assumed to be symphonic map,  we get 
$$
\begin{aligned}
\begin{cases}
		&\Delta^{h_f^*g_{N}}\left(\cos \varphi(t) f(x)\right)+\frac{\Lambda}{c^2}\cos \varphi(t) f(x)=0 \label{a1}\\
	&\Delta^{h_f^*g_{N}}\left(\sin \varphi(t) \right)+\frac{\Lambda}{d^2} \sin \varphi(t)=0\label{a2}
\end{cases}
\end{aligned}$$
then
\begin{equation}\label{ee}
	\begin{aligned}
	&\frac{c^2\cos ^2 \varphi(t)}{a^4\cos ^4 t} \Delta^{f^* g_{\mathbb{S}^{n-1} }}|_{\mathbb{S}^{m_1}}\cos\varphi(t) f+\frac{c^2\cos ^2 \varphi(t)}{b^4\sin ^4 t} \Delta^{f^* g_{\mathbb{S}^{n-1}}}|_{\mathbb{S}^{m_2}}\cos\varphi(t) f \\
	& +\frac{\partial}{\partial t}\left(\frac{1}{h^2(t)}\varphi^{\prime}(t)^2 \frac{\partial}{\partial t}\right)(\cos\varphi(t) f)-\left(m_1 \frac{\sin t}{\cos t}-m_2 \frac{\cos t}{\sin t}\right)\frac{1}{h^2(t)} \varphi^{\prime}(t)^2 \frac{\partial}{\partial t}(\cos\varphi(t) f)\\	
	&+\frac{\Lambda}{c^2}\cos \varphi(t) f=0. 
	\end{aligned}
\end{equation}
Meanwhile, 
\begin{equation}\label{ee1}
	\begin{aligned}
		&\frac{\partial}{\partial t}\left(\frac{1}{h^2(t)}\varphi^{\prime}(t)^2 \frac{\partial}{\partial t}\right)(\sin\varphi(t) f)-\left(m_1 \frac{\sin t}{\cos t}-m_2 \frac{\cos t}{\sin t}\right)\frac{1}{h^2(t)} \varphi^{\prime}(t)^2 \frac{\partial}{\partial t}(\sin\varphi(t) f)\\	
		&+\frac{\Lambda}{d^2}\sin \varphi(t)=0. 
	\end{aligned}
\end{equation}
since $ f $ is symphonic map, by lemma \ref{eq9},  we can infer that  $ c^2\sin(\varphi(t)) \times $ \eqref{ee}-$ d^2\cos(\varphi(t)) \times $\eqref{ee1}=0 等价于
	\begin{equation}\label{eq2}
	\begin{split}
		&\frac{k^2(\varphi)}{h^2(t)}(\varphi^{\prime}(t)^3)^{\prime}-	\frac{k^2(\varphi)}{h^2(t)}\left(m_1 \frac{\sin t}{\cos t}-m_2 \frac{\cos t}{\sin t}-2\frac{h^\prime(t)}{h(t)}\right)\varphi^{\prime}(t)^3+ \frac{2kk^\prime}{h^2(t)}\varphi^{\prime}(t)^4\quad   \\
		+	&\left(\frac{c^4\cos ^2 \varphi(t)}{a^4\cos ^4 t} \|f_1^* {g_{\mathbb{S}^n}}\|^2 +\frac{c^4\cos^2 \varphi(t)}{b^4 \sin ^4 t} \|f_2^{*} {g_{\mathbb{S}^n}}\|^2 \right)  \sin \varphi(t)\cos \varphi(t)=0
	\end{split}
\end{equation}
Simplifying the equation $ c^2 \cos \varphi \times$ \eqref{ee}+$ d^2 \sin \varphi \times$ \eqref{ee1}=0 we can get \eqref{df1}.

Now we prove (3), set 
	$$\begin{aligned}
L(\var, \var^\prime, t)=\left(\frac{k^4(\varphi(t))}{h^4(t)}\varphi^{\prime}(t)^4+a_1 \frac{\sin^4 \varphi(t)}{\cos ^4 t}+a_2 \frac{\sin ^4 \varphi(t)}{\sin ^4 t}\right) \cos ^{m_1} t \sin ^{m_2} t
\end{aligned}$$
By 	$$\begin{aligned}
	\frac{d}{dt}\left(\frac{\partial L}{\partial \var^\prime}\right)-\frac{\partial L}{\partial \var}=0. 
\end{aligned}$$
we get (3). 
\end{proof}
Now we study the existence of Hopf map.  for the map $h_f$ given by \eqref{eq3}, direct computation shows that 
\begin{equation*}
	\begin{split}
		\| dh_f\|^2=  \frac{c^2\cos^2\varphi(t)}{a^2\cos^2(t)} |df_1 |^2+\frac{c^2\cos^2\varphi(t)}{b^2\sin^2(t)} |df_2 |^2+\frac{\kappa^2(\varphi)}{h^2(t)}\varphi^\prime(t)^2
	\end{split}
\end{equation*}
where $ f_1(\cdot)= f_1(\cdot, y), f_2(\cdot)= f_1(x, \cdot) $,   $ x \in \mathbb{S}^{m_{1}},  y \in \mathbb{S}^{m_{2}} $. 

we define the following functional
\begin{equation*}
	\begin{split}
		J(\varphi)=\frac{1}{4}\int_0^{\frac{\pi}{2}}\left(\frac{k^4(\varphi(t))}{h^4(t)}\varphi^{\prime}(t)^4+a_1 \frac{\sin^4 \varphi(t)}{\cos ^4 t}+a_2 \frac{\sin ^4 \varphi(t)}{\sin ^4 t}\right) \cos ^{m_1} t \sin ^{m_2} t d t
	\end{split}
\end{equation*}
where  $a_1, a_2$ are given in \eqref{cc-1}.   $\operatorname{Vol}(A)$ denotes the volume of  $A$,  $ \mathcal{F}\left(f_1\right)=\int_M |f_1^* g_{\mathbb{S}^n}|^2 $,  $ \mathcal{F}\left(f_2\right)=\int_M |f_2^* g_{\mathbb{S}^n}|^2 $
In fact,  for the metric $g$ given by \eqref{metric},  $
g^\prime=\left(c^2 \cos ^2 s\right) g_{\mathbb{S}^{n_1}}+\left(d^2 \sin ^2 s\right) g_{\mathbb{S}^{n_2}}+k^2(s) \mathrm{d} s^2
$. 
For map $\left(Q^{m_1+m_2+1}(a,  b), g\right)\to \left(N, g_N\right) $,  simple computation shows
$$\begin{aligned}
	J(\varphi)=\int_{Q^{m_1+m_2+1}(a,  b)}\norm{ h_f^* g_N}^2	\mathrm{d} \text{v}_g
\end{aligned}$$

Similar to\lemref{lem-9},  we get 

\begin{lem} For function $ \varphi(t) $ and the map $ h_f $ given in \eqref{eq3},  so
	\begin{equation*}
		\begin{aligned}
			F(h_f)=ab\operatorname{Vol}(\mathbb{S}^{m_1}) \operatorname{Vol}(\mathbb{S}^{m_2})J(\varphi)
		\end{aligned}
	\end{equation*}
\end{lem}
\begin{rem}
Therefore,  symphonic map is also the critical poit of the functional  $J(\varphi)$.  Thus, In order to  find the  ellipsoid join map,  just need  to find the  critical point  of $J(\varphi)$. . 
\end{rem}

 For the space $X$ given by \eqref{space},  we set
	$$\begin{aligned}
	\mathcal{A}=\{\var \in X  \mid  \var(0)=0, \varphi(\frac{\pi}{2})=\frac{\pi}{2} \}
\end{aligned}$$

A similar argument as in \lemref{33} gives that 
\begin{lem}
	There exists minimizer of $J(\varphi)$ in 	$\mathcal{A}$ which $\widehat{\varphi}(\mathrm{t})$ satisifies Euler-Lagrange equation (\cref{eq2})
\end{lem}

	\begin{rem}
		Here we cannot get similar conclusion as Corollary \autoref{r-1}, \autoref{r-2}, 
		\autoref{r-3}, \autoref{r-4}. 
	\end{rem}

\begin{rem}
Only one order ODE is needed to handle with.  
\end{rem}
\begin{proof}
It proceeds as in  \autoref{33}. 
\end{proof}

By   Sobolev embedding of  $\mathrm{W}^{1, 4}\left(\left(0,  \frac{\pi}{2}\right)\right)$ into $\mathrm{C}^0\left(\left[0,  \frac{\pi}{2}\right]\right)$,   we know that $\widehat{\varphi}(t) \in C^0\left(\left[0,  \frac{\pi}{2}\right]\right)$. 
\begin{lem}
	$\widehat{\varphi} \in \mathrm{C}^{\infty}\left(\left(0,  \frac{\pi}{2}\right)\right)$
\end{lem}

 \begin{proof}
  	Set $\psi(t)=k(\widehat{\varphi})^2\left(\widehat{\varphi}^{\prime}(t)\right)^3$.  By \lemref{dd} ,  we have 
 	$$
 	\psi^{\prime}(t)+A(t) \psi(t)+B(t)=0
 	$$
 where
 	$$
 	\begin{aligned}
 		& A(t)=  \frac{k(\alpha)^2}{h^2(t)}\left(m_1 \frac{\sin t}{\cos t}-m_2 \frac{\cos t}{\sin t}-2\frac{h^\prime(t)}{h(t)}\right) \\
 		& B(t)=-h^2(t)\sin \widehat{\varphi}(t) \cos \widehat{\varphi}(t)\left(\frac{c^4\cos ^2 \widehat{\varphi}(t)}{b^4\sin ^4 t} \|f_1^* g_{\mathbb{S}^{n-1}}\|^2 +\frac{c^4\cos ^2 \widehat{\varphi}(t)}{a^4\cos ^4 t}\| f_2^* g_{\mathbb{S}^{n-1}} \|^2\right) .  \\
 		&
 	\end{aligned}
 	$$
 The remained proof is as  in \autoref{ccc}.  
 \end{proof}
 
%   Take and fix any $t_0 \in\left(0,  \frac{\pi}{2}\right)$.  For any $t \in\left(0,  \frac{\pi}{2}\right)$,  we have
%$$
%\begin{aligned}
%	\left|\psi(t)-\psi\left(t_0\right)\right| & =\left|\int_{r_0}^t \psi(t) d t\right| \\
%	& =| \int_{t_0}^t(A(t) \psi(t)+B(t)) d t \mid \\
%	& \leq\left|\int_{t_0}^t\left(|A(t)|\left|\hat{\psi}^{\prime}(t)\right|^3+|B(t)|\right) d t\right| \\
%	& \leq\left\{\int_{t_0}^t|A(t)|^4 d t\right\}^{\frac{1}{4}}\left\{\int_{t_0}^t\left|\vec{\varphi}^{\prime}(t)\right|^4 d t\right\}^{\frac{3}{4}}|+| \int_{t_0}^t|B(t)| d t \mid \\
%	& \longrightarrow 0 \text { as } t \longrightarrow t_0 . 
%\end{aligned}
%$$
%since $\hat{\varphi} \in \mathrm{L}^{1, 4}\left(\left(0,  \frac{\pi}{2}\right)\right)$,  and $\int_{t_0-\varepsilon}^{t_0+\varepsilon}|A(t)|^4 d t<\infty$ and $\int_{t_0-\varepsilon}^{t_0+\varepsilon}|B(t)|^4 d t<\infty$.  Thus we have $\psi \in \mathrm{C}^0\left(\left(0,  \frac{\pi}{2}\right)\right)$.   We use a bootstrap argument for the equality
%$$
%\psi^{\prime}(t)=-A(t) \psi(t)-B(t)
%$$
%%Since the right hand side of $(13)$ is of class $\mathrm{C}^0,  \psi(t)$ is of class $\mathrm{C}^{-1}$,  therefore $\hat{\phi}(t)$ is of class $\mathrm{C}^2$.  Since the right hand side of $(13)$ is of class $\mathrm{C}^{\mathrm{l}},  \psi(t)$ is of class $\mathrm{C}^2$,  therefore $\bar{\psi}(t)$ is of class $\mathrm{C}^1$.  lnductively
%bootstrap argument  implies that $\hat{\varphi}(t)$ is of class $\mathrm{C}^{\infty}$. 
%The above lemmas imply the existence of Hopf construction. 

\begin{thm}\label{thm4. 1}For map $f: \mathbb{S}^{m_{1}} \times \mathbb{S}^{m_{2}} \rightarrow \mathbb{S}^{n-1}$,  The Hopf map given by (\ref{eq3}) exists, and is also symohonic map. 
\end{thm}
\bibliographystyle{acm}	
\bibliography{J:/myonlybib/myonlymathscinetbibfrom2023, J:/myonlybib/low-quality-bib-to-publish} 
	\end{document}